\documentclass[final,onefignum,onetabnum]{siamart171218}
\usepackage{todonotes}

\usepackage[square,sort,comma,numbers]{natbib}
\setcitestyle{authoryear,open={(},close={)}}

\usepackage{lipsum}
\usepackage{amsfonts}
\usepackage{graphicx}
\usepackage{epstopdf}
\usepackage{algorithmic}
\usepackage{amstext,amssymb}
\usepackage{float} 
\usepackage{amsopn}
\usepackage{array} 
\usepackage[caption=false]{subfig}
\usepackage{wrapfig}

\ifpdf
  \DeclareGraphicsExtensions{.eps,.pdf,.png,.jpg}
\else
  \DeclareGraphicsExtensions{.eps}
\fi

\newsiamremark{remark}{Remark}
\newsiamremark{hypothesis}{Hypothesis}
\crefname{hypothesis}{Hypothesis}{Hypotheses}
\newsiamthm{claim}{Claim}

\headers{Aposteriori Error Estimates for CutFEM with Boundary Correction}{}

\title{A Posteriori Error Estimates with Boundary Correction for a Cut Finite Element Method 
}

\author{Erik Burman
	\thanks{Department of Mathematics, 
	University College London, Gower Street, London, UK--WC1E  6BT, United Kingdom 
  	(\email{e.burman@ucl.ac.uk})}
\and 
	Cuiyu He
	\thanks{Department of Mathematics, University College London, Gower Street, London, 
	UK--WC1E  6BT, United Kingdom
  	(\email{c.he@ucl.ac.uk})}
\and
	Mats G. Larson
	\thanks{Department of Mathematics and Mathematical Statistics, Ume\aa \hspace{1mm}University,
	SE-90187 Ume\aa, Sweden (\email{mats.larson@umu.se})}
	}

\def\IR{\mathbb R}

\newcommand{\bfn}{\boldsymbol n}

\newcommand{\bfx}{\boldsymbol x}

\newcommand{\bfp}{\boldsymbol p}

\newcommand{\bn}{\boldsymbol n}

\newcommand{\bfnu}{\boldsymbol \nu}

\newcommand{\mcF}{\mathcal{F}}

\newcommand{\tn}{|\mspace{-1mu}|\mspace{-1mu}|}

\newcommand{\DnF}{{D_{{\boldsymbol n}_F}}}

\numberwithin{equation}{section}
\newtheorem{lem}{Lemma}[section]

\newtheorem{thm}{Theorem}[section]
\newtheorem{rem}{Remark}[section]

\newtheorem{example}{Example}[section]

\newcommand{\jump}[1]{[\![#1]\!]}

\def\cT{{\mathcal T}}
\def\cE{{\mathcal E}}

\def\O{{\Omega}}
\def\o{{\omega}}

\begin{document}

\maketitle

\begin{abstract}
In this work we study a residual based a posteriori error estimation
for the CutFEM method applied to an elliptic model problem. We
consider the problem with non-polygonal boundary and the analysis
takes into account the geometry and data approximation on the
boundary. The reliability and efficiency are theoretically
proved. Moreover, constants are robust with respect to how the domain
boundary cuts the mesh.
\end{abstract}



\section{Introduction}

Meshing and re-meshing procedures remain an important challenge in
computational methods, since they can be  very computationally
expensive for the finite element methods, particularly for problems
with complex geometries and interfaces that move during the
computational process. Several methods have been invented to
alleviating the meshing procedure of the exact domain. Fictitious
domain methods were introduced in the finite element context in 
the papers by Glowinski et al. \cite{GP92} and unfitted finite
element methods in the works by Barrett and Elliott \cite{BE84}. 

Since these seminal works many approaches have been suggested on how to
integrate the geometry data in finite element computations in a way
that reduces the meshing effort. For instance the fat boundary method
by Bertoluzza et al. \cite{BIM05}, the fictitious domain methods using Lagrange
multipliers inspired by extended finite element methods, pioneered by
Haslinger and Renard \cite{HR09,
  BH10a}, or the cut finite element method using Nitsche's method,
introduced by Hansbo and Hansbo and further developed in the fictitious
domain framework by various authors
\cite{HH02,Bu10, EDH10, BBH11, BH12}.

Cut finite element methods have been extensively studied in the recent
years, both in the context of interface problems and for
fictitious domain methods \cite{BCHLM15}. We restrict the discussion herein to the
fictitious domain case. The
methodology alleviates the meshing process by employing a background mesh
that can be highly structured and letting the domain boundary cut
through the elements of the mesh. In cells cut by the boundary the equations are integrated
only on the intersection of the element with the physical domain. The
boundary conditions are then imposed either using Lagrange multipliers
or Nitsche's method and stability with respect to the cut may be
ensured using a ghost penalty stabilization \cite{Bu10}. The
method was originally designed for continuous approximation spaces, but
has been adapted for various other methods, e.g, Discontinuous
Galerkin method \cite{JoLa13, GM18}, Hybrid High order Methods
\cite{BE19} or Isogeometric analysis \cite{ELL18}.
It has been applied to a number of different partial differential
equations, e.g, incompressible elasticity/Stokes' equations
\cite{BuHa14, MLLR14, BCM15, GO18}, linear elasticity \cite{hansbo2017cut},
Helmholtz equations \cite{Swift18},
time dependent parabolic problems on moving domains \cite{HLZ15}, Oseen's
problem \cite{MSW18, WSMW18} and other fluid models \cite{SW14, SSKW16}. It has also been applied successfully to shape optimization
problems \cite{BEHLL17, BEHLL18, BWB18} and other advanced engineering applications \cite{BTB19}.

Typically, in cut finite element methods, a domain with curved
boundary is approximated using a piecewise affine boundary
approximation. This gives a sufficiently good geometry approximation
for piecewise affine elements, but if higher order elements are used
the geometry approximation must be improved. To alleviate the
integration problem resulting from the elements cut by curved
boundaries, isoparametric techniques \cite{Lehr17} and so called 
boundary value correction techniques have been proposed 
\cite{BHL18, BBCL18}. For both these cases
optimal order a priori error estimates for arbitrary order of the
polynomial approximation have been derived. For an analysis of the
fitted finite element method on domains with curved boundaries we
refer to \cite{BK94}.

The purpose of this paper is to design an a posteriori error estimation
for the cut finite element method for problems with non-polygonal
boundary. We will base our discussion on the cut finite element
methods for the Poisson problem introduced by Burman and Hansbo
\cite{BH12}. This method uses Nitsche's method \cite{Nit71} to impose
Dirichlet boundary conditions and a ghost penalty term \cite{Bu10} to 
enhance stability in the boundary zone. We restrict the discussion to 
the fictitious domain problem and piecewise affine approximation. 
The main motivation for studying the a posteriori  error estimation is  for the application of adaptive mesh refinement (AMR) procedure. It is well known that AMR is extremely useful for problems with singularities, discontinuities, sharp derivatives, etc. and it has been extensively studied in the last two decades, \cite{verfurth1994posteriori, ainsworth2011posteriori}. However, there is very limited work in the literature that takes the geometry approximation into account. 
In \cite{dorfler1998adaptive} the a posteriori error estimation is
studied for the conforming finite element method on curved boundary
where the boundary vertices of the approximated mesh must be located
on the true boundary and it is assumed that data can be requested at
any point on the boundary and inside the domain, i.e., there is only
approximation of the geometry. In \cite{ainsworth2017computable} a
fully computable error bound is provided for the conforming linear
elements with pure Neumann data. To the author's knowledge there is no
existing work on the a posteriori error estimation for cut finite
element methods on problems with curved boundary, to allow for. Another approach for
the handling of geometric singularities in the cutFEM framework was
recently proposed in \cite{JLL18}.

In our error analysis we do not assume that the vertices of the
approximation boundary are located on the boundary of the continuous
problem, while we do require that the discrete boundary is a
sufficiently good
approximation of the boundary of the continuous problem. We do not
require that data inside the domain can be everywhere requested but
only those originally provided in the continuous setting. For the
formulation it is necessary to extend the source term data from the
domain of the continuous problem to the computational domain, the
extension must satisfy certain stability properties.
 The error estimator comprises two parts. One part is due to the
 numerical approximation; and the other part is exclusively due to the
 boundary approximation. We refer to this latter contribution as the
 boundary correction error. This part must be computed using a locally
 improved boundary approximation in the boundary zone. The computation
 of this correction is discussed in the numerical section.

Unlike the methods whose mesh is an exact
 partition of the computational domain, one major challenge for cut
 finite element is that the approximated domain cuts the background
 mesh in an arbitrary way. Because of this the classical efficiency
 analysis, i.e., by applying the
 local elementary bubble functions, is not robust. Therefore we divide the efficiency
 part of the estimates in two parts. The residuals in the bulk, away
 from the boundary, are estimated in the ordinary fashion. For the
 residuals in the cut elements and the nonconforming ghost penalty
 operator, we instead prove a global lower bound of best approximation type,
 showing that the boundary residuals are bounded by the best
 approximation of $u$ in the physical domain and oscillations.
 
 
This paper is organized as follows. In \cref{sec:2} the model problem and the cut finite element method is introduced. The error estimator is introduced in \cref{sec:3} and as well its global reliability. The local efficiency is proved in  \cref{sec:4}. Finally, we show the  results of several numerical experiment in \cref{sec:5}.

\section{Model Problem and the Cut Finite Element Method}\label{sec:2}
\subsection{The Continuous Problem}
Let $\Omega$ be a domain in $\mathbb{R}^d$ with Lipschitz continuous,
piecewise smooth boundary $\partial \Omega$ and exterior unit normal $\bfn$. 

We consider the problem: find $u:\Omega \rightarrow \IR$ such that
\begin{alignat}{2}\label{eq:poissoninterior}
-\Delta u &= f \qquad 
&& \text{in $\Omega$}
\\ \label{eq:poissonbc}
u &= g \qquad && \text{on $\partial\Omega$}
\end{alignat}
where $f\in L^2(\Omega)$ and $g\in H^{1/2}(\partial \Omega)$ 
are given data. It follows from the Lax-Milgram lemma that there 
exists a unique solution $u \in H^1(\Omega)$ to this problem. We also 
have the  following elliptic regularity estimate
\begin{equation}\label{eq:ellipticregularity}
\|u\|_{H^{s+2}(\Omega)} \lesssim \|f\|_{H^s(\Omega)}, \qquad 
-1 \leq s \leq s_0
\end{equation} 
for some $s_0\geq 1/2$ depending on the domain. Here and below we use the notation 
$\lesssim$ to denote less or equal up to a generic constant that is independent of the
mesh-geometry configuration.

\subsection{The Mesh, Discrete Domains, and Finite Element Spaces}
Assume that
$\partial \Omega$ \textcolor{black}{is composed of a finite number of
  smooth surfaces $\Gamma_i$, such that $\partial \Omega = \cup_i \bar
  \Gamma_i$}. 
We let $\rho$ 
be the signed distance function, negative on the inside and 
positive on the outside, to $\partial \Omega$ and we let 
$U_\delta(\partial \Omega)$, for $\delta>0$,
be the tubular neighborhood $\{\bfx \in \IR^d : |\rho(\bfx)| < \delta\}$ 
of $\partial \Omega$.  Let $\Omega_0 \subset \IR^d$ be a polygonal domain such 
that $U_{\delta_0}(\Omega) \subset \Omega_0$ 
where $\delta_0$ is chosen such that $\rho$ is well defined in $U_{\delta_0}(\Omega)$. 
Let $\cT_{0,h}$ 
be a partition of $\Omega_0$ into shape 
regular triangles or tetrahedra. Note that this setting allows meshes with locally  
dense refinement.

Given a subset $\omega$ of $\Omega_0$, let 
$\cT_h(\omega)$ be the submesh defined by
\begin{equation}
\cT_{h}(\omega) = \{K \in \cT_{0,h} : \overline{K} \cap 
\overline{\omega}
\neq \emptyset \}
\end{equation}
i.e., the submesh consisting of elements that intersect 
$\overline{\omega}$, and let 
\begin{equation}
\triangle_{h}(\omega)= \bigcup\limits_{K \in \cT_{h}(\omega)}K 
\end{equation}
be the union of all elements in $\cT_h(\omega)$. 


For each $\cT_{0,h}$ let $\Omega_h$ be a polygonal domain approximating $\Omega$ and we assume that
$\partial \Omega_h \subset U_{\delta_0}(\partial \Omega)$ which implies that $\partial \Omega_h$ is within the distance of $\delta_0$ to $\partial \Omega$. We assume neither $\Omega_h \subset \Omega$ nor  $\Omega \subset
\Omega_h$, instead the maximum distance between the two domains has to
be small enough. 

Let the active mesh be defined by
\begin{equation}
\cT_{h} := \cT_h( \O \cup \Omega_h) 
\end{equation}
i.e., the submesh consisting of elements that intersect 
$\Omega_h\cup\Omega$, and let 
\begin{equation}
\triangle_h := \triangle_h(\Omega\cup \Omega_h)
\end{equation}
be the union of all elements in $\cT_h$. Since $\partial \O_h$ does not necessarily fit 
the mesh we denote by $\cT_h^b$ the set of elements that are ``cut" by $\partial \O_h$,
\[
\cT_h^b = \{ K \in \cT_{0,h}: K \cap \partial \O_h \neq \emptyset\} \subset \cT_h.
\]
\textcolor{black}{We assume that $\Omega_h$ is constructed in such a 
way that for each $K  \in \cT_{h}^b$, the intersection $K\cap \partial \Omega_h$ is a subset of 
a $d-1$ dimensional hyperplane, i.e., a line segment in two dimensions and a subset of a plane 
in three dimensions.} \textcolor{black}{Under the smallness assumption of $h$, for any element
  $K \in \cT_h^b$ there
exists an element $K' \in \cT_h( \Omega_h)\setminus \cT_h^b$ such that
$\mbox{dist}(K,K') = O(h)$ where $O(\cdot)$ denotes the big ordo. This is always possible for small enough $h$
since $\Omega$ is Lipschitz.}  

Also we denote by $\cE$ the set of all facets on $\cT_h$ and by $\cE_I \subset \cE$ the set of all interior facets with respect to $\cT_h$.
For each $F \in \cE$ denote by $\bfn_F$ an unit vector normal to $F$ and by $h_F$ the diameter of $F$.
For each $K \in \cT_h$ denote by $h_K$ the diameter of $K$ and by $\cE_K$ the set of all facets of $K$.

On the boundary $\partial \O_h$ let $\bfn_h$ be the outer normal to $\partial \O_h$.
For each $\Omega_h$ we assume that, for $\delta_0$ small enough, there exist functions 
$\bfnu_h:\partial \Omega_h \rightarrow \mathbb{R}^d$, $|\bfnu_h|=1$, and $\varrho_h:\partial \Omega_h \rightarrow \mathbb{R}$, on $U_{\delta_0}(\partial \Omega)$, such that the function 
$\bfp_h(\bfx,\varsigma):=\bfx + \varsigma \bfnu_h(\bfx)$, is well
defined and satisfies
$\bfp_h(\bfx,\varrho_h(\bfx)) \in \partial \Omega$ for all $\bfx
\in \partial \Omega_h$. 
The existence of the vector-valued function $\bfnu_h$ is known to hold
on Lipschitz domains, see Grisvard \cite{Gris11}.

We further assume that 
$\bfp_h(\bfx,\varsigma)
\in U_{\delta_0}(\Omega)$ for all $\bfx \in \partial \Omega_h$ and all
$\varsigma$ between $0$ and $\varrho_h(\bfx)$.
For conciseness we will drop the second argument, $\varsigma$, of $\bfp_h$ below whenever it takes the value $\varrho_h(\bfx)$ and let $\bfp_h$ denotes the map $\bfp_h:\partial \Omega_h \rightarrow \partial \Omega$.
Moreover, we assume that the following assumption is satisfied
\begin{equation}\label{eq:geomassum-a}
\| \varrho_h \|_{L^\infty(\partial \Omega_h \cap K)} \le  O(h_K) \quad
\forall \, K \in \cT_h^b.
\end{equation}
The above assumption immediately implies that $\partial \O_h$ is within the distance of $O(h)$ of $\partial \O$. More precisely, 
\begin{equation}
\| \varrho \|_{L^\infty(\partial \Omega_h \cap K)} = O(h_K) \quad
\forall \, K \in \cT_h^b.
\end{equation}
This assumption is necessary for the constant in the a posteriori
error estimates to be independent of the geometry/mesh
configuration. It is however not enough to guarantee optimal a priori
error estimates, which requires $\| \varrho \|_{L^\infty(\partial \Omega_h \cap K)} \le O(h_K^2)$, 
and $\| \bfn - \bfn_h \|_{L^\infty(\partial \Omega_h \cap K)} \le O(h_K)$, 
see \cite{BHL18}.


\subsection{The Cut Finite Element Method} 
In this subsection we recall the CutFEM introduced in \cite{BH12}. We begin with some 
necessary notation. Let $\mcF_h$ be the set of faces intersecting the approximate 
boundary $\O_h$:
\[ 
\mcF_{h} = \{ F \in \cE_I \,:\, (K_F \cup K_F') \cap \partial \O_h \neq \emptyset\}
\]
where $K_F$ and $K'_F$ are those two elements sharing $F$ as a common facet, and 
for any discontinuous function $v$ define its jump on the facet $F$ by
\[
\jump{v}|_F := v^+_F - v^-_F \quad \mbox{and} \quad 
v_F^{\pm}(x)
= \lim \limits_{s \rightarrow 0^+} 
v(\bfx \mp s \bfn_F).
\]
Next define the finite element space
\begin{equation}
V_{h} :=\{ v \in H^1(\triangle_h): v|_K \in \mathbb{P}_1(K) \quad \forall \, K \in \cT_{h} \}
\end{equation}
and the forms
\begin{equation}\label{forms}
\begin{split}
a_0(v,w) &:= (\nabla v,\nabla w)_{\Omega_h} 
- \left<D_{\bfn_h} v,w\right>_{\partial \Omega_h} 
- \left<v, D_{\bfn_h} w\right>_{\partial \Omega_h} 
+   \sum_{K \in \cT_h^b} \dfrac{\beta}{h_K}\left< v,w\right>_{\Gamma_K},\\
j_h(v,w) &:= \gamma\sum_{F \in \mcF_{h}}  h_F \left< \jump{D_{\bfn_F} v},\jump{D_{\bfn_F} w}\right>_F,\\
a_h(v,w)& := a_0(v,w) + j_h(v,w), 
\\
l_h(w) &:= (f,w)_{\Omega_h} 
- \left<g_h , D_{\bfn_h} w\right>_{\partial \Omega_h} 
+ \sum_{K \in \cT_h^b} \dfrac{\beta}{h_K}\left<g_h,w\right>_{\Gamma_K},
\end{split}
\end{equation}
where $D_{\bfn_h} v := \bfn_h \cdot \nabla v$,  $\gamma$ \textcolor{black}{and $\beta$} are positive constants, $g_h$ is an approximation of $g$ defined on $\partial \Omega_h$,
typically $g_h(\bfx) = g \circ \bfp_h$, and
$\Gamma_K = K \cap \partial \O_h$ \textcolor{black}{and $f\vert_{\Omega_h \setminus
  \Omega}$ is defined by some suitable extension (in the low order
case considered here it can be taken to
be zero).} 
\begin{remark}
The stabilizing term $j_{h}(v,w)$, which is the so called ghost penalty term, is introduced to extend the coercivity of $a_h(\cdot,\cdot)$ to all of $\triangle_h$, see  \cite{Bu10,MasLarLog14}. Thanks to this property one may prove that the condition
  number of the linear system is uniformly bounded independent of how $\Omega_h$ is
  oriented compared to the mesh.
\end{remark}

\begin{remark}
In order to guarantee 
    the coercivity of \cref{forms} $\beta$ has to be chosen large enough.
\end{remark}


The finite element method is then to find  $u_h \in V_h$ such that 
\begin{equation}\label{eq:fem}
a_h(u_h, v) = l_h(v) \quad \forall \,v \in V_h
\end{equation}
where $a_h$ and $l_h$ are defined in \cref{forms}. 

For $v \in H^1(\O_0)$ define the continuous and discrete energy norm
respectively by
\begin{align}\label{def:ene_norm}
\tn v \tn^2 &:= 
\|\nabla v\|^2_{\Omega} 
+\|D_{\bfn}v \|^2_{H^{-1/2}(\partial \Omega)} 
+ \|h^{-\frac12} v\|^2_{\partial \Omega}
\end{align}
and
\begin{align}\label{def:ene_norm_h}
\tn v \tn_h^2 &:= 
\|\nabla v\|^2_{\Omega_h} 
+\sum_{K \in \cT_h^b} h_K \|D_{\bfn_h} v \|^2_{\Gamma_K} 
+ \sum_{K \in \cT_h^b}h_K^{-1}\|v\|^2_{\Gamma_K}.
\end{align}

\subsection{Some Important Inequalities}
We have the following well known inverse and trace inequalities \cite[Section 1.4.3]{DiPE12},
\begin{equation}\label{eq:standard_trace}
\|v\|_{\partial K} \le C( h_K^{-1/2}\|v\|_{K} + h_K^{1/2}\|\nabla
v\|_K) \quad \forall \, v \in H^1(K),
\end{equation}
\begin{equation}\label{eq:discrete_trace}
h_K^{-\frac12} \|v_h\|_{\partial K} + \|\nabla v_h\|_K \le C h_K^{-1}\|v_h\|_{K} 
\quad \forall \, v_h \in \mathbb{P}_1(K)
\end{equation}
and \cite{HH02}
\begin{equation}\label{eq:boundary_trace}
	\|v_h\|_{\Gamma_K} \le C( h_K^{-1/2}\|v_h\|_{K} +
        h_K^{1/2}\|\nabla v_h\|_K) \quad \forall \, v_h \in \mathbb{P}_1(K)
\end{equation}
where  the constant $C$ is independent of the relative location of $\O_h$.

\begin{lem}
Let $v \in H_0^1(\O)$. Then for any $K$ such that $K \in \cT_h^b$ or $K \cap (\O \setminus \O_h) \neq \emptyset$ there exists a local convex neighborhood $\mathcal{S}_K$ of $K$ such that $v$ vanishes 
on a nonzero subset of $\partial S_K$ and
\begin{equation}\label{poincare}
	\|v\|_{K} \lesssim h_K\| \nabla v\|_{\mathcal{S}_K},
\end{equation}
where we defined $v$ outside $\O$ using the trivial extension $v\vert_{\O^c} = 0$.
\end{lem}
\begin{proof}
If $K \subset \O^c$, the estimate \cref{poincare} is obvious since $v$ is taken uniformly zero 
outside $\O$. If  $K \cap \partial \O \neq \emptyset$, then we let $x\in K \cap \partial \O$ and 
let $S_K = B_\delta(x)$ be the open ball centered at $x$ with radius $\delta \sim h$. The estimate 
(\ref{poincare}) is a direct consequence of the Poincar\'e inequality on $S_K$.
Otherwise $K\subset \O$. From \cref{eq:geomassum-a} it follows that 
\[
\mbox{dist}(K, \partial \O) \leq O(h)
\]
and thus there is $x \in \partial \Omega$ and $\delta \sim h$ such that $K \subset B_\delta(x)$ and  now \cref{poincare} follows with $S_K = B_\delta(x)$ from the Poincar\'e inequality on $S_K$. This completes the proof of the lemma.
\end{proof}

\section{Global Reliability}\label{sec:3}
Let $e := u - u_h$ and $\tilde e$ be any function such that
\[
	\tilde e \in H^1(\O) \quad \mbox{and} \quad \tilde e|_{\partial \O} = (g-u_h)
\] 
where $u_h$ is the solution to \cref{eq:fem}.

For each element $K \in \cT_h$, define the local element error indicator $\eta_K$ by
\begin{equation} \label{indicator}
	\eta_K^2 =   h_K^2\| f \|_{K \cap \O_h}^2
	+
	 \sum_{F \in \cE_K \cap \cE_I}\dfrac{h_F}{2} \|\jump{\DnF u_h}\|_{F}^2
	+
	h_K^{-1} \|g_h-u_h\|_{L^2(\Gamma_K)}^2 
	+  \| \nabla \tilde e\|_{K \cap \O}^2 .
\end{equation}
and let the global error estimator be defined by
\begin{equation} \label{estimator}
\eta = \left( \sum_{K \in \cT_h} \eta_K^2   \right)^{1/2}.
\end{equation}

\begin{thm}
Let $\tilde e \in H^1(\O)$ such that $\tilde e|_{\partial \O} = (g - u_h)|_{\partial \O}$.
We have the following reliability bound
\begin{equation}\label{reliability-result}
\| \nabla e\|_{\O} \le C_r
\left(  \sum_{K \in \cT_h} \eta_K^2
+  \sum_{K \in \cT_h}h_K^2 \|f\|_{(\O \setminus \O_h) \cap K}^2 \right)^{1/2}
\end{equation}
where the constant $C_r$ does not depend on the location of domain-mesh intersection nor the mesh size.
\end{thm}

\begin{proof}
Note that $e - \tilde e \in H_0^1(\O)$. By the triangle inequality we have that
\begin{equation}
\begin{split}
	\| \nabla e\|_{\O} &\le \| \nabla (e - \tilde e)\|_{\O}+ \| \nabla \tilde e\|_{\O} 
	\\
	 &=
	\sup_{ v \in H_0^1(\O)} \dfrac{(\nabla ( e - \tilde e), \nabla v)_{\O}}{\| \nabla v\|_{\O}}
	+ 
	 \| \nabla \tilde e\|_{\O} \\[2mm]
	 &\le
	 	\sup_{ v \in H_0^1(\O)} \dfrac{(\nabla e , \nabla v)_{\O}}{\| \nabla v\|_{\O}}
	+
	2 \| \nabla \tilde e\|_{\O}.
\end{split}
\end{equation}
It then suffices to show that
	\[
		\sup_{ v \in H_0^1(\O)} \dfrac{(\nabla e , \nabla v)_{\O}}{\| \nabla v\|_{\O}} 
		\lesssim
		\left(  \sum_{K \in \cT_h} \eta_K^2 
		+  \sum_{K \in \cT_h}h_K^2 \|f\|_{(\O \setminus \O_h) \cap K}^2 \right)^{1/2},
	\]
which is a direct result from \cref{{lem:error-rep}} and \cref{lem:reliability-bound} below. This completes the proof of the theorem.
\end{proof}

\begin{rem}
	We omit the second term in \cref{reliability-result} in the
        algorithm computation in order to avoid integration on the curved domain. What's more, one can easily make it a higher order term by satisfying
	\[
	\|\varrho\|_{L^{\infty}(\Gamma_K)} \le o(h_K) \quad \forall \, K \in \cT_h^b.
	\]
\end{rem}

\begin{lem}\label{lem:error-rep}
For any  $v \in H_0^1(\O)$ and $v_h \in V_h$
the following equality holds:
\begin{equation} \label{error-rep}
	(\nabla e, \nabla v)_\O = \mathcal{A}_1 + \mathcal{A}_2
\end{equation}
where
\begin{equation}
	 \mathcal{A}_1 = 
	 (f , v - v_h)_{\O_h}
	+(f , v)_{\O \setminus \O_h} 
\end{equation}
and
\begin{align}
	\mathcal{A}_2
	&=- \dfrac{1}{2}\sum_{K \in \cT_h} \sum_{F \in \cE_K}
		  \int_{F \cap {\O_h}} \jump{\DnF u_h} (v-v_h) \,ds
		  + \left< g_h-u_h,D_{\bfn_h} v_h\right>_{\partial \Omega_h} 
		 \\
	&\qquad  \nonumber
	    -\dfrac{1}{2}\sum_{K \in \cT_h} \sum_{F \in \cE_K}
		  \int_{F \cap {(\O \setminus \O_h)}} \jump{\DnF u_h} v \,ds
		  \\
	&\qquad  \nonumber	  
	- \sum_{K \in \cT_h^b} \dfrac{\beta}{h_K} \left<g_h -u_h,v_h\right>_{\Gamma_K}
	+j_h(u_h,v_h) .
\end{align}
\end{lem}
\begin{proof} Extending $v\in H^1_0(\Omega)$ to $\IR^d$ by setting $v=0$ outside of $\Omega$ and using integration by parts gives
\begin{equation} \label{error_rep_1}
\begin{split}
	 (\nabla e, \nabla v)_{\O} &=
	(f ,v)_{\O}
	- (\nabla u_h, \nabla v)_{\O_h}  
	- (\nabla u_h,  \nabla v)_{\O \setminus \O_h} 
	 \\[2mm]
	&=( f, v)_\O 
	-  (\nabla u_h , \nabla v_h )_{\O_h} 
	- (  \nabla u_h , \nabla (v -v_h) )_{\O_h} 
	- (\nabla u_h,  \nabla v)_{\O \setminus \O_h}.
\end{split}
\end{equation}
For the second term in \cref{error_rep_1}, by \cref{eq:fem} we have
\begin{equation} \label{error_rep_2}
\begin{split}
	-(\nabla u_h, \nabla v_h)_{\O_h} &=
	- \left<D_{\bfn_h} u_h,v_h\right>_{\partial \Omega_h} +
	 j_h(u_h,v_h)
	-(f,v_h)_{\Omega_h}  
	\\
	&\qquad + \left< g_h-u_h, D_{\bfn_h} v_h\right>_{\partial \Omega_h}
	- \sum_{K \in \cT_h^b} \dfrac{\beta}{h_K}\left<g_h -u_h,v_h\right>_{\Gamma_K}. 
\end{split}
\end{equation}
For the last two terms in (\ref{error_rep_1}) applying the integration by parts gives
\begin{equation}\label{error_rep_3}
\begin{split}
&- (\nabla u_h ,\nabla (v -v_h) )_{\O_h}  
	-(\nabla u_h , \nabla v )_{\O \setminus \O_h} \\
&\qquad =
	  - \sum_{K \in \cT_h} \int_{\partial (K \cap \O_h)} (D_{\bfn} u_h) (v - v_h) \,ds
	  -\sum_{K \in \cT_h} \int_{\partial (K \cap (\O \setminus \O_h))} (D_{\bfn} u_h) v \,ds
	\\
&\qquad =
	-\dfrac{1}{2}\sum_{K \in \cT_h} \sum_{F \in \cE_K}
	\int_{F \cap {\O_h}} \jump{\DnF u_h} (v-v_h) \,ds
	\\
	&\qquad \qquad 
	 -\dfrac{1}{2}\sum_{K \in \cT_h} \sum_{F \in \cE_K}
	\int_{F \cap {(\O \setminus \O_h)}} \jump{\DnF u_h} v \,ds
	 +\left< D_{\bfn_h} u_h, v_h \right>_{\partial \O_h}.
\end{split}
\end{equation}
In the last equality we used the following identity
\[
	-\int_{\partial \O_h} \nabla u_h \cdot \bfn v \,ds 
	-\int_{\partial (\O \setminus \O_h)} \nabla u_h \cdot \bfn v \,ds = 0
\]
thanks to the fact that $v = 0$ on $\partial \O$ and in $\O_h \setminus \O$, where $\bfn$ without subscript denotes the outer normal to the boundary of domain being integrated. The desired 
identity \cref{error-rep} is then a direct consequence of \cref{error_rep_1}--\cref{error_rep_3}.
This completes the proof of the lemma.
\end{proof}

\begin{remark}
Note that $\mathcal{A}_1$ and $\mathcal{A}_2$ in \cref{lem:error-rep} does not contain the inconsistency error caused by the boundary approximation. In other words, the boundary correction error has been independently isolated by the term $\|\nabla \tilde e\|$. Since any $\tilde e$ satisfying the boundary condition will yield an upper bound for the error, an inappropriate construction of $\tilde e$ will potentially over estimate the error.
\end{remark}
 
 \begin{lem} \label{lem:reliability-bound}
 Let $\mathcal{A}_1$ and $\mathcal{A}_2$ be given in \cref{lem:error-rep}. Then we have the following estimates:
 \begin{equation}\label{bound_for_A1}
 \mathcal{A}_1 \lesssim  \left( \sum_{K \in \cT_h} \left(h_K^2 \|f\|_{K \cap \O_h}^2  +
 h_K^2 \|f\|_{K \cap ( \O \setminus \O_h)}^2\right)
 \right)^{1/2}
 \|\nabla v\|_{\O},
 \end{equation}
 and
\begin{equation} \label{bounds-for-A2}
	\mathcal{A}_2 \lesssim
	\left( 
		\sum_{K \in \cT_h} \sum_{F \in \cE_K \cap \cE_I} 
		\dfrac{h_F}{2} \|\jump{\DnF u_h}\|_{F}^2 
		+\sum_{K \in \cT_h^b} h_K^{-1} \|g_h-u_h\|_{\Gamma_K}^2 \right)^{1/2} 
		 \|\nabla v\|_{\O}.
\end{equation}
 \end{lem}
\begin{proof}
The first term of $\mathcal{A}_1$ in (\ref{error-rep})
can be bounded directly using the Cauchy-Schwarz inequality and the approximation property
 of the Scott-Zhang interpolation \cite{scott1990finite},
 \begin{equation} \label{residual-1}
 \begin{split}
 	(f , v - v_h)_{\O_h} 
	&\le
	\sum_{K \in \cT_h} \| f \|_{K \cap \O_h} \| v -v_h\|_{K}
\\
&	\lesssim
	\sum_{K \in \cT_h} h_K\| f \|_{K \cap \O_h} \| \nabla v\|_{\triangle_K} \\
	&\lesssim
	\left(\sum_{K \in \cT_h} h_K^2\| f \|^2_{K \cap \O_h} \right)^{1/2} \| \nabla v\|_{\O},
\end{split}
 \end{equation}
 where $\triangle_K$ is the union of elements in $\cT_h$ that shares at least one vertex with $K$.

By \cref{poincare} the second term of $\mathcal{A}_1$ can now be bounded as follows:
  \begin{equation} \label{residual-2}
  \begin{split}
	(f , v)_{\O \setminus \O_h} &
	\le
	 \sum_{K \in \cT_h} \| f\|_{K \cap (\O \setminus \O_h)} \|v\|_{K \cap (\O \setminus \O_h)}
\\
&	 \lesssim 
	 \left(
	\sum_{K \in \cT_h} h_K^2\|f\|_{K \cap (\O \setminus \O_h)}^2 \right)^{1/2} \| \nabla v\|_{\O}.
\end{split}
\end{equation}
(\ref{bound_for_A1}) is then a direct result of (\ref{residual-1}) and (\ref{residual-2}).
	  
We now proceed to bound terms in $\mathcal{A}_2$.
The first term of $\mathcal{A}_2$ can be bounded directly by applying the Cauchy-Schwarz inequality and the approximation property of the Scott-Zhang interpolation,
\begin{equation} \label{bd-term1}
\begin{split}
	&
	 -	 \sum_{K} \sum_{F \in \cE_K} \int_{F\cap \Omega_h}
		\jump{\DnF u_h}(v -v_h)\,ds 
\\
&\qquad \le	 \sum_{K} \sum_{F \in \cE_K} 
		\|\jump{\DnF u_h}\|_{F\cap \Omega_h} \|v -v_h \|_{F} 
		\\
	&\qquad \lesssim
		 \sum_{K} \sum_{F \in \cE_K} 
		h_K^{1/2}\|\jump{\DnF u_h}\|_{F\cap \Omega_h} \|\nabla v \|_{\Delta_K}\\
	&\qquad \lesssim
	\left( \sum_{K} \sum_{F \in \cE_K} h_K  \|\jump{\DnF u_h}\|_{F\cap \Omega_h}^2 \right)^{1/2} \| \nabla v\|_{\O}.
\end{split}
\end{equation}
To bound the second term in $\mathcal{A}_2$ we apply the trace inequality and \cref{poincare} that
\begin{equation} \label{bd-term2}
\begin{split}
		    &-\sum_{K \in \cT_h} \sum_{F \in \cE_K}
		  \int_{F \cap {(\O \setminus \O_h)}} \jump{\DnF u_h} v \,ds 
		  \\
&\qquad		  \le
		    \sum_{K \in \cT_h} \sum_{F \in \cE_K}
		  \|\jump{\DnF u_h}\|_{F \cap {(\O \setminus \O_h)}} 
		  \|v\|_{F}
		   \\
		    &\qquad \lesssim
		    \sum_{K \in \cT_h} \sum_{F \in \cE_K}
		  \|\jump{\DnF u_h}\|_{F \cap {(\O \setminus \O_h)}} 
		 \left( h_K^{-1/2} \|v\|_{K}  + h_K^{1/2}\|\nabla v\|_{K} \right)
		\\
		 &\qquad \lesssim
		  \sum_{K \in \cT_h} \sum_{F \in \cE_K} h_K^{1/2}
		  \|\jump{\DnF u_h}\|_
		  {F \cap {(\O \setminus \O_h)}} \|\nabla v\|_{\mathcal{S}_K} \\
		  &\qquad \lesssim
		  \left(  \sum_{K \in \cT_h} \sum_{F \in \cE_K} h_K
		  \|\jump{\DnF u_h}\|^2_
		  {F \cap {(\O \setminus \O_h)}} \right)^{1/2} \|\nabla v\|_{\O}.
\end{split}
\end{equation}
The jump penalty term in $\mathcal{A}_2$ can be bounded by applying the Cauchy-Schwarz, triangle, trace, and the inverse inequalities and the stability  
of the interpolator:
\begin{equation} \label{bd-term3}
\begin{split}
	 j_{h}(u_h,v_h) &\le 
	  \gamma\sum_{F \in \mcF_{h}}  h_F \|\jump{ \DnF u_h}\|_{F} 
	  \|\jump{\DnF v_h}\|_{F}
	  \\
	 & \lesssim
	  \sum_{F \in \mcF_{h}} h_F^{1/2}\|\jump{ \DnF u_h}\|_{F} 
	 \left (\sum_{K \in \{K_F^+ , K_F^-\}}h_F^{1/2}\|  \nabla v_h|_K\|_{F} \right)
	\\
	 &\lesssim \sum_{F \in \mcF_{h}} h_F^{1/2} \|\jump{ \DnF u_h}\|_{F} 
	 \left(
	 \sum_{K \in \{K_F^+ , K_F^-\}} \| \nabla v_h\|_{K}  \right)
	 \\
	  &\lesssim \sum_{F \in \mcF_{h}} h_F^{1/2}\|\jump{\DnF u_h}\|_{F} 	  
	  \| \nabla v\|_{ \triangle_F}
	  \\
	  &
	  \lesssim
	  \left( \sum_{F \in \mcF_{h}} h_F \|\jump{\DnF u_h}\|_{F}^2  \right)^{1/2} \| \nabla v\|_{\O}
	  ,
\end{split}
\end{equation}
where  $K_F^+ (K_F^-)$ is the element with $\bfn_F (-\bfn_F)$ being its outer unit normal on $F$ and $\triangle_F := \triangle_{K_F^+} \cup \triangle_{K_F^-}$.

The remaining two terms in $\mathcal{A}_2$ can be bounded by applying the Cauchy-Schwarz and trace inequalities, \cref{poincare} and the stability of the interpolator:
\begin{equation} \label{bd-term5}
\begin{split}
	& \sum_{K \in \cT_h^b} \dfrac{\beta}{h_K}\left<g_h-u_h, v_h \right>_{\Gamma_K}
	+
	\left< g_h-u_h, D_{\bfn_h} v_h\right>_{\partial \Omega_h} \\
	&\qquad \lesssim
	\sum_{K \in \cT_h^b} \left( h_K^{-1} \|g_h-u_h\|_{\Gamma_K} \	\|  v_h\|_{\Gamma_K}
	+
	\|g_h-u_h\|_{\Gamma_K} \|D_{\bfn_h} v_h\|_{\Gamma_K}\right)
	\\
	&\qquad \lesssim
		\sum_{K \in \cT_h^b} h_K^{-1} \|g_h-u_h\|_{\Gamma_K} 
		\left( h_K^{-1/2}\|  v_h\|_{K} +  h_K^{1/2}\|\nabla  v_h\|_{K}  \right)
		\\
	&\qquad \lesssim
	\sum_{K \in \cT_h^b} h_K^{-1/2} \|g_h-u_h\|_{\Gamma_K} 
	\|  \nabla v\|_{\mathcal{S}_K \cup \Delta_K}
\\
&\qquad	\lesssim
	\left( \sum_{K \in \cT_h^b} h_K^{-1} \|g_h-u_h\|^2_{\Gamma_K}  \right)^{1/2} \|\nabla v\|_{\O}.
\end{split}
\end{equation}

Finally, combining (\ref{bd-term1})-(\ref{bd-term5}) and the Young's inequality yields (\ref{bounds-for-A2}). This completes the proof of the lemma.
\end{proof}

\section{Efficiency}\label{sec:4}
In this section we prove the efficiency of the error indicator introduced in \cref{indicator}. 
For the classical finite element method the efficiency is usually proved in a local fashion 
by applying the local facet and element bubble functions. For elements that are not intersected with the boundary and whose facets do not belong to the ghost penalty set we are able to prove the efficiency using the classical bubble technique.

\begin{lem}\label{lem:effi-for-normal-elements}
Let $K$ be a given element in $ \cT_h$ such that $K \not \in \cT_h^b$ and $\cE_K \cap \mathcal{F}_h = \emptyset$.
Then the following local efficiency result holds:
\begin{equation}
\eta_K \le C_e \| \nabla (u-u_h)\|_{\tilde\Delta_K},
\end{equation}
where $\tilde \Delta_K$
is a local neighborhood of $K$ and the efficiency constant $C_e$ does not depend on the mesh size nor the domain-mesh intersection.
\end{lem}
\begin{proof}
	The proof is classical and we refer \cite{verfurth1994posteriori,cai2017residual}.
\end{proof}

\begin{remark}
For the regular elements, we take the boundary correction error $\|\nabla \tilde e\|_{K \cap \O}$ to be $0$ since we can always design $\tilde e$ in such a way that it vanishes inside regular elements in order to avoid over-estimation (see \cref{subsec:boundary-error-computation}). 
\end{remark}

\begin{figure}[h]\label{fig:triangle-demo}
\centering
\includegraphics[width=0.40\textwidth]{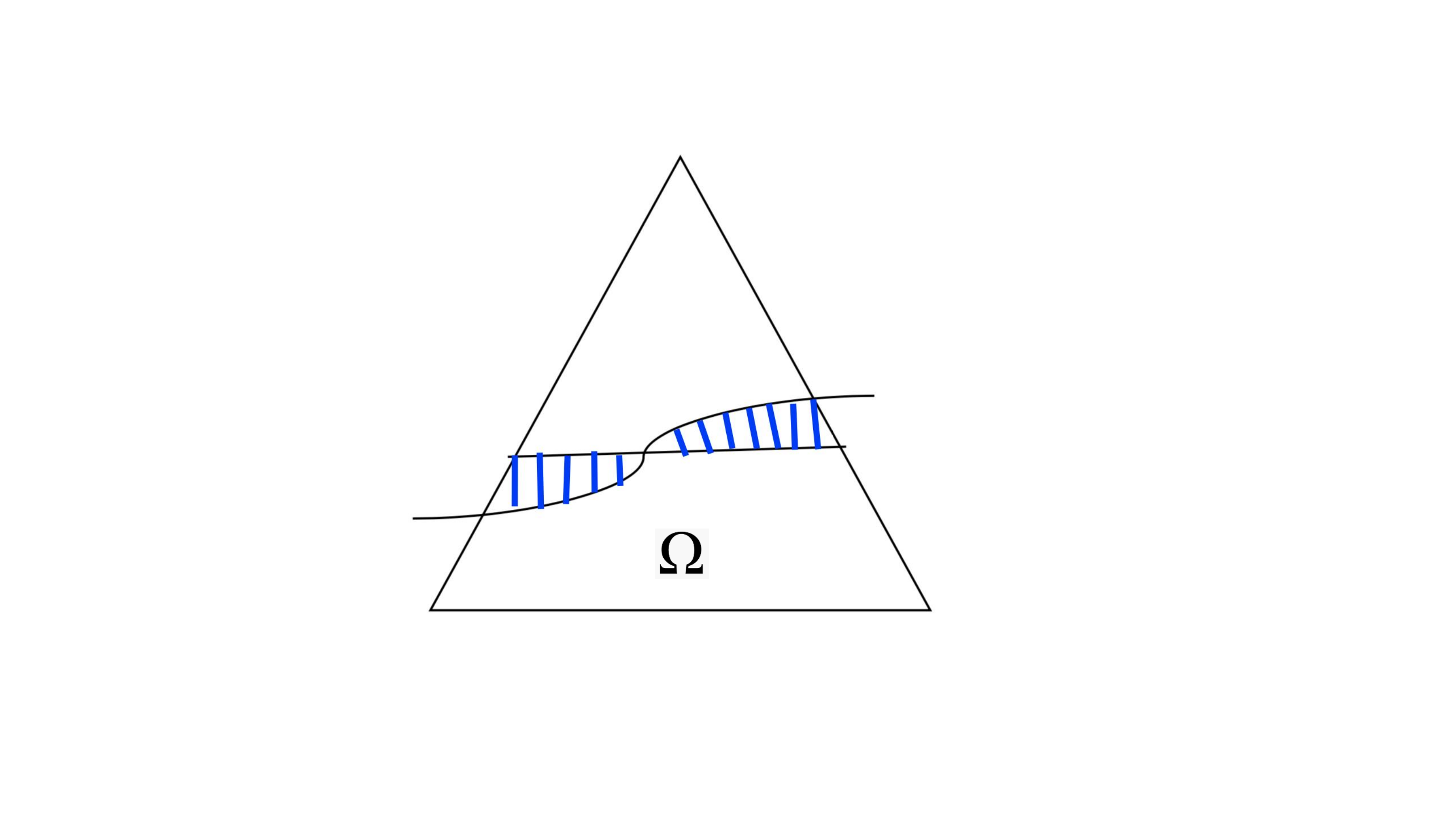}
\caption{An irregular element $K$ and $D_K$ (the shaded area)}
\end{figure}
For an element $K$ that is irregular, i.e., $K \in \cT_h^b$ or $\cE_K \cap \mathcal{F}_h \neq \emptyset$, applying the same bubble technique for cut elements unfortunately will result in the dependence on the location of domain-mesh intersection. 
As an alternative we prove the efficiency for the term of ghost penalty as a whole. The efficiency bounds for the element residual on irregular elements can be then also proved with the aid of the efficiency result for the ghost penalty and the numerical scheme.

For each $K$ define $E_K = \partial K \cap ((\Omega \cup \Omega_h)
\setminus (\Omega \cap \Omega_h))$ and  $D_K = K \cap ((\Omega \cup \Omega_h) \setminus (\Omega \cap \Omega_h))$, see \cref{fig:triangle-demo} for an illustration.

We first prove a discrete Poincare inequality that will be useful in the following efficiency proof.
\begin{lem} \label{lem:boundary-region-bound}
For any $v \in V_h$ the following estimate is true:
      \begin{equation}\label{f-0}
      \sum_{K \in \cT_h} (h_K^{-2}\| v\|_{D_K}^2 + h_K^{-1} \| v\|_{E_K}^2) \lesssim 
       \tn v \tn^2_h + \tn v_h\tn^2_{j_h},
  \end{equation} 
where $\tn v_h \tn_{j_h}^2 = j_h(v_h, v_h)$.
\end{lem}
\begin{proof}

First, for each $K$ such that $E_K, D_K \neq \emptyset$, note that
using the trace inequality \cref{eq:discrete_trace} we see that
\[
h_K^{-1} \| v\|_{E_K}^2 \leq h_K^{-1} \| v\|_{\partial K}^2 \lesssim h_K^{-2} \| v\|_{K}^2.
\]
It follows that
\[
h_K^{-2}\| v\|_{D_K}^2 + h_K^{-1} \| v\|_{E_K}^2 \lesssim h_K^{-2} \| v\|_{K}^2.
\]
\textcolor{black}{Now observe that there exists} a set denoted by $S_K \subset \triangle_h$ such that
$E_K, D_K \subset S_K$, $\mbox{diam}(S_K \cap \O_h)$ is of order $h_K$
and  for which using Poincar\'e's inequality
\begin{equation}\label{f-1}
	\| v\|_{K} \lesssim h_K\| \nabla v\|_{S_K} + 
	\sum_{K' \cap S_K \neq 0} h_{K'}^{1/2} \|v\|_{\Gamma_{K'}}.
\end{equation}
 By the equivalence of norms theorem we also have
  \begin{equation}\label{f-2}
	\| \nabla v\|_{S_K} \le \| \nabla v\|_{S_K \cap \O_h} + 
	\sum_{\substack{ F \in \cE_I \\F \subset S_K}}
		h_F^{1/2}\| \jump{\DnF v}\|_{F}.
\end{equation}
Combining \cref{f-1} and \cref{f-2}
gives \cref{f-0}. This completes the proof of the lemma.
\end{proof}

We assume that for each $K \in \cT_h^b$ there exists at least one vertex, say $z$, of $K$ such that
$\mbox{diam}(\o_z \cap \O_h) = O(h_K)$ where $\o_z$ is the union of all all elements sharing $z$ as a vertex. 
Define
\[
 \mbox{osc}(f) = \sum_{K} \left(h_K^2\| f - f_z\|_{\o_z \cap \O_h}^2  + 
 h_K^2\|f - f_z\|_{K \cap (\O \setminus \O_h)}^2 \right)^{1/2},
\]
where $f_z$ is chosen such that 
$f_z = \mathop{\mathrm{argmin}}_{c \in \mathcal{R}} h_K\| f -
c\|_{\o_z \cap \O_h}$. We will now proceed and prove a best
approximation result for the element residuals of the cut
elements. First we give a Lemma showing that the element residual of a cut element
can be bounded by the ghost penalty term, the boundary residual and
oscillation in data. 
\begin{lemma}\label{lem:f-1}
For each $K \in \cT_h^b$ its element residual has the following bound:
\begin{equation}\label{efficiency-f}
	\begin{split}
		 h_K\| f \|_{K \cap \O_h} \lesssim&
	\sum_{F \subset \bar\o_z \cap \cE_I} h_F^{1/2}\| \jump{\DnF u_h} \|_{F \cap \O_h}
	+ \left(\sum_{K \subset \o_z} h_K^{-1}\| u_h - g_h\|^2_{\Gamma_K}\right)^{\frac12}\\
	&+
	h_K\| f - f_z\|_{\o_z \cap \O_h} + h_K\|f - f_z\|_{K \cap (\O \setminus \O_h)}.
	\end{split}
\end{equation}

\end{lemma}
\begin{proof}
By the triangle inequality we firstly have
	\begin{equation}\label{split-f}
		h_K\| f \|_{K \cap \O_h} 
		\le
		  h_K\|f - f_z\|_{K \cap \O_h }  +
		h_K\| f_z\|_{\o_z \cap \O_h},
	\end{equation}
	where $f_z$ at this point can be chosen as any arbitrary constant.
Let $\lambda_z$ be the barycentric hat function associated with $z$ then we also have
\begin{equation} \label{split}
	\|f_z\|^2_{\o_z \cap \O_h}  \lesssim \|f_z \lambda_z^{1/2}\|_{\O_h}^2
	= (f, f_z\lambda_z)_{\O_h} - (f - f_z, f_z\lambda_z)_{\O_h}.
\end{equation}
Let $w_z = f_z \lambda_z$. Applying \cref{eq:fem} and integration by parts yields
\begin{equation}\label{estimate-of-f-fz}
	\begin{split}
	&(f, f_z\lambda_z)_{\O_h} =
	(\nabla u_h,\nabla w_z)_{\Omega_h} 
	- \left<D_{\bfn_h} u_h,w_z\right>_{\partial \Omega_h} 
	- \left<u_h - g_h, D_{\bfn_h} w_z\right>_{\partial \Omega_h} \\
	&+   \sum_{K' \in \cT_h^b} \dfrac{\beta}{h_{K'}}\left< u_h- g_h,w_z\right>_{\Gamma_{K'}}
	+
	\gamma\sum_{F \in \mcF_{h}}  h_F \left< \jump{ \DnF u_h},
	 \jump{ \DnF w_z}\right>_{F}\\
	=& 
	\sum_{F \subset \o_z} \left< \jump{\DnF u_h}, w_z \right>_{F \cap \O_h}
	- \left<u_h - g_h, D_{\bfn_h} w_z\right>_{\partial \Omega_h} \\
	&+   \sum_{K' \in \cT_h^b} \dfrac{\beta}{h_{K'}}\left< u_h- g_h,w_z\right>_{\Gamma_{K'}}
	+
	\gamma\sum_{F \in \mcF_{h}}  h_F \left< \jump{ \DnF u_h},
	 \jump{ \DnF w_z}\right>_{F \cap \bar \o_z}\\
	\lesssim&
	\sum_{F \subset \o_z} h_F^{-1/2}\| \jump{\DnF u_h} \|_{F \cap \O_h} \|f_z\|_{\o_z \cap \O_h} 
	+ 
	\sum_{ \substack{K' \subset \o_z \\ K' \in \cT_h^b}} h_{K'}^{-3/2}\| u_h - g_h\|_{\Gamma_{K'}} \|f_z\|_{\o_z \cap \O_h} \\
	& +
	\sum_{ \substack{F \in \mcF_{h}\\ F \subset \bar\o_z }}  h_F^{-1/2} \| \jump{\DnF u_h} \|_{F \cap \bar \o_z} \|f_z\|_{\o_z \cap \O_h}.
	\end{split}
\end{equation}
The last inequality utilizes the fact that $f_z$ is a constant and the trace and inverse inequalities.
Combining \cref{split}, \cref{estimate-of-f-fz} and the Cauchy Schwarz inequality gives
\begin{equation}
	\begin{split}
	h_K\| f_z \|_{\O_h \cap K} \lesssim&
	\sum_{F \subset \bar\o_z \cap \cE_I} h_F^{1/2}\| \jump{\nabla u_h\cdot \bfn_F} \|_{F}
\\ &\qquad 
	+ \left(\sum_{K' \subset \o_z} h_{K'}^{-1}\| u_h - g_h\|^2_{\Gamma_{K'}}\right)^{\frac12}
	+	\| f - f_z\|_{\o_z \cap \O_h},
	\end{split}
\end{equation}
 which, combining with \cref{split-f}, yields \cref{efficiency-f}.
\end{proof}

We now prove the main bound for the residuals in the cut elements.

\begin{thm}
Let $u$ and $u_h$ be the solution to \cref{eq:poissoninterior} and \cref{eq:fem}, respectively.
Then the following best approximation result holds:
	\begin{equation}\label{effi-for-irregular}
	\begin{split}
	&j_h(u_h, u_h) + \sum_{K \in \cT_h^b} h_K^2 \|f\|_{K \cap \O_h}^2 \\
	\le& C_e 
	 \inf_{v_h \in V_h} \left( \tn u - v_h \tn^2 + \tn v_h\tn_{j_h}^2 + \sum_{K \in \cT_h^b} h_{K}^{-1}\| v_h - g_h\|^2_{\Gamma_{K}}+\mbox{osc}(f)^2
	 \right).
	\end{split}
	\end{equation}
	where the constant $C_e$ does not depend on the mesh size nor
        on the domain-mesh intersection.
\end{thm}
\begin{proof}
By the triangle inequality we have
\begin{equation}\label{two-parts}
	j_h(u_h,u_h) \lesssim j_h(u_h - v_h, u_h - v_h) + j_h(v_h,v_h)
\end{equation}
and
\begin{equation}\label{bound-parts}
\sum_{K \in \cT_h^b}h_K^{-1}\|u_h -
  g_h\|_{\Gamma_K}^2 \lesssim 
\sum_{K \in \cT_h^b}h_K^{-1}\|u_h -
  v_h\|_{\Gamma_K}^2 
  + \sum_{K \in \cT_h^b}h_K^{-1}\| v_h -g_h\|_{\Gamma_K}
\end{equation}
Denote by $\phi =u_h - v_h $. Then
\[
 j_h(u_h - v_h, u_h - v_h) + \sum_{K \in \cT_h^b}h_K^{-1}\|u_h -
  v_h\|_{\Gamma_K}^2  \leq \tn \phi\ \tn_h^2 + \tn \phi \tn_{j_h}^2.
\]
Applying the coercivity result (see \cite{BCHLM15}), \cref{forms} and \cref{eq:fem} gives
\begin{equation}\label{4-parts}
\begin{split}
	&\tn \phi \tn_h^2 + \tn \phi \tn_{j_h}^2 \lesssim j_h(u_h - v_h, \phi) + a_0(u_h - v_h, \phi) 
\\	
	&\qquad = l_h(\phi) -  a_0(v_h, \phi) - j_h(v_h, \phi)
\\
	&\qquad = l_h(\phi) - (f, \phi)_{\Omega}  + (\nabla u,  \nabla \phi)_{\Omega}  -\left<D_{\bfn} u , \phi \right>_{\partial \Omega}
	-  a_h(v_h, \phi) - j_h(v_h, \phi) 
\\
&\qquad \triangleq
	\sum_{i =1}^4 \mathcal{A}_i,
\end{split}
\end{equation}
where
\begin{align*}
	\mathcal{A}_1 &= (f, \phi)_{\O_h} - (f,\phi)_{\O}, 
\\
	\mathcal{A}_2 &= \left<v_h - g_h, \nabla \phi \cdot \bn \right>_{\partial \Omega_h} 
	-
	 \sum_{K \in \cT_h^b} \dfrac{\beta}{h_K} \left< v_h - g_h, \phi \right>_{\Gamma_K},
\\
	\mathcal{A}_3 &= (\nabla u,  \nabla \phi)_{\Omega} -  (\nabla v_h, \nabla \phi)_{\Omega_h}
	+
	 \left< D_{\bfn_h}v_h, \phi\right>_{\partial \Omega_h} -
	 \left< D_{\bfn}u, \phi \right>_{\partial \Omega}, 
	 \\
	 \mathcal{A}_4 &= -j_h(v_h, \phi).
\end{align*}
To estimate $\mathcal{A}_1$, applying the Cauchy-Schwarz inequality and \cref{lem:boundary-region-bound} gives
\begin{equation}\label{A1-estimtate}
\begin{split}
	\mathcal{A}_1 &\le
	\left(  \sum_{K \in \cT_h} h_K^2\|f\|_{D_K}^2 \right)^{1/2}
	\left(  \sum_{K \in \cT_h} h_K^{-2}\| \phi\|_{D_K}^2 \right)^{1/2}\\
	&\lesssim \left(  \sum_{K \in \cT_h} h_K^2\|f\|_{D_K}^2 \right)^{1/2} 
	\left( \tn \phi \tn_h + \tn \phi\tn_{j_h} \right).
\end{split}
\end{equation}
$\mathcal{A}_2$ can be estimated directly using the Cauchy-Schwarz inequality,
\begin{equation}\label{A-2 estimate}
\begin{split}
	\mathcal{A}_2 &\le   
	\left(   \sum_{K \in \cT_h^b}  h_K^{-1} \| v_h - g_h \|_{\Gamma_K}^2 	\right)^{1/2}
        \left( \sum_{K \in \cT_h^b}  h_K\| D_{\bfn_h} \phi\|_{\Gamma_K}^2
        +
         \sum_{K \in \cT_h^b}  h_K^{-1} \| \phi \|_{\Gamma_K}^2 \right)^{1/2}\\
	& \lesssim
	\left(   \sum_{K \in \cT_h^b}  h_K^{-1} \| v_h - g_h \|_{\Gamma_K}^2 	\right)^{1/2}
	\tn  \phi \tn_h 
\end{split}
\end{equation}
To estimate $\mathcal{A}_3$, add and subtract suitable terms to
obtain, 
\begin{equation*}
\begin{split}
\mathcal{A}_3 &=
	(\nabla (u - v_h),  \nabla \phi)_{\Omega} 
	-
	 \left< \nabla (u -v_h ) \cdot \bn, \phi \right>_{\partial \Omega}\\
&\qquad + (\nabla v_h, \nabla \phi)_{\Omega \setminus \Omega_h} -  (\nabla v_h, \nabla \phi)_{\Omega_h \setminus \Omega}
	 +
	 \left< \nabla v_h\cdot \bn, \phi \right>_{\partial \Omega_h}
	 -\left< \nabla v_h\cdot \bn, \phi \right>_{\partial \Omega}.
	\end{split}
\end{equation*}
Observe that using integration by parts we have
\begin{equation*}
\begin{split}
&(\nabla v_h, \nabla \phi)_{\Omega \setminus \Omega_h} -  (\nabla v_h,
\nabla \phi)_{\Omega_h \setminus \Omega} 
+
	 \left< \nabla v_h\cdot \bn, \phi \right>_{\partial \Omega_h}
	 -\left< \nabla v_h\cdot \bn, \phi \right>_{\partial \Omega} 
	 \\
&\qquad  \leq \sum_{K \in
           \mathcal{T}_h^b} \int_{E_K}
         |\jump{\DnF v_h} \phi| \,ds \\
&\qquad  \leq j_h(v_h,v_h)^{\frac12} 
\left( \sum_{K \in \mathcal{T}_h^b}
\|h^{-\frac12} \phi\|^2_{E_K}
\right)^{\frac12}\\
&\qquad \leq j_h(v_h,v_h)^{\frac12} (\tn  \phi \tn_h + \tn  \phi\tn_{j_h})
\end{split}
\end{equation*}
\textcolor{black}{where 
  we used Lemma \ref{lem:boundary-region-bound} for the last inequality}.
Then applying the Cauchy-Schwarz inequality and \cref{lem:boundary-region-bound}  
gives
\begin{equation}\label{A3-aa}
\begin{split}
 \mathcal{A}_3\le& 
	  \|\nabla (u - v_h)\|_{\Omega} \| \nabla \phi\|_{\O}
	  +\| D_{\bn_h} (u -v_h)\|_{H^{-1/2}(\partial \Omega)} \|\phi\|_{H^{1/2}(\partial \O) }\\
	  &+j_h(v_h,v_h)^{\frac12} (\tn  \phi \tn_h + \tn  \phi\tn_{j_h})
          \\
\le & (\|\nabla (u - v_h)\|_{\Omega} +\| D_{\bn_h} (u
-v_h)\|_{H^{-1/2}(\partial \Omega)} +j_h(v_h,v_h)^{\frac12}) (\tn  \phi \tn_h + \tn  \phi\tn_{j_h}).
\end{split}
\end{equation}
Here we also used the inequality 
\[
	\|\phi\|_{H^{1/2}(\partial \O) } \lesssim
	\|\phi\|_{H^1(\Omega)} \lesssim \tn \phi\tn_h + \tn \phi\tn_{j_h}.
\]
Collecting the bounds  (\ref{4-parts})--(\ref{A3-aa}) we see that
\begin{equation}\label{eq:jump_bound}
\tn \phi \tn_h + \tn \phi \tn_{j_h} \lesssim \inf_{v_h \in V_h} (\|\nabla (u -
  v_h)\|_{\Omega} + \| D_{\bn_h} (u -v_h)\|_{H^{-1/2}(\partial
    \Omega)} + j_h(v_h,v_h)^{\frac12})
\end{equation}
 
 Finally, combining \cref{lem:f-1}, \cref{two-parts}, \cref{bound-parts} and \cref{eq:jump_bound} yields 
 \cref{effi-for-irregular}.
This completes the proof of the theorem.
\end{proof}

\begin{remark}
Regarding the efficiency for the  indicator  of the boundary correction error,
we have the following efficiency result:
{\color{black}
\[
	\inf_{\substack{ \tilde e \in H^1(\O) \\ \tilde e = g - u_h \mbox{ on } \partial \O}} \| \nabla \tilde e\|
	= \| u - u_h\|_{H^{1/2}(\partial \O)} \lesssim
	\tn u - u_h\tn.
\]} 
This also indicates that $\tilde e$ needs to be designed properly to avoid over estimation. In \cref{sec:5} we propose an approach to 
design and compute $\tilde e$.
\end{remark}

\section{Numerical Results}\label{sec:5}
In this section we present several numerical examples to validate the performance of the a posteriori error estimator in the adaptive mesh refinement procedure.  The adaptive mesh refinement procedure is set as follows:
\[
	\mbox{Solve} \rightarrow \mbox{Estimate} \rightarrow  \mbox{Mark} \rightarrow \mbox{Refine} \rightarrow \mbox{Solve}.
\]
For the penalty parameters in the finite element method, we set $\beta =10$ and $\gamma =0.1$. For the refinement strategy, we use the D\"orfler marking strategy and the refine rate is set to be ten percent.
Regarding the domain approximation,
for a given $\phi$ being the level set function that satisfies $\phi=0$ on the boundary and negative (positive) inside (outside) the domain $\O$, let $\phi_h$ be the nodal interpolation of $\phi$ with respect to $\cT_h$. And we define 
\begin{equation}\label{domain-approx}
	\partial \O_h = \{ \bfx: \phi_h(\bfx) =0\}.
\end{equation} 

\subsection{Computation of $\| \nabla \tilde e\|_{\O}$}\label{subsec:boundary-error-computation}
In this subsection we introduce one way to compute the  boundary correction error  $\| \nabla \tilde e\|_{\O}$.
We firstly construct a boundary correction mesh,  denoted by $\cT_h^{bc}$, which is finer than the current mesh $\cT_h$ in order to more accurately track the boundary $\partial \O$. More precisely, the new mesh is built inside the union of all intersection elements, i.e., 
\[
	\cup_{K \in \cT_h^{bc}} K \subset \cup_{K \in \cT_h^b}  K.
\]

For each element $K \in \cT$, without loss of generality,  in two dimensions, we assume that $\phi(z_0) \le \phi(z_1) \le \phi(z_2)$ where $z_i, i =0, 1,2$ are vertices of $K$. 
For each $K \in \cT_h^b$
we assume the following must be true: there exists at least one facet in $K$ such that the nodes connected to that facet have both positive and negative values, i.e.,
\[ \phi(z_0) <0 \quad \mbox{and} \quad \phi(z_2)>0. \]
We then further partition the elements in $\cT_h^b$ based on the intersection of the domain and the mesh. The pseudo code for the algorithm is provided in \cref{alg:buildtree}. An example boundary correction mesh for Example 1 is shown in \cref{fig:example-bc-mesh}.

\begin{algorithm}[h!] 
\caption{Build the boundary correction mesh in two dimensions}
\label{alg:buildtree}
\begin{algorithmic}
\FOR{cell $K \in \cT_h^b$ }
	\IF{ $\phi(z_1) > 0$}
		\IF{ $\phi((z_1 + z_2)/2) \ge 0$}
		\STATE{Set Type as 'a'}
		\STATE{Find intersection points $z_3$--$z_5$ and form 2 triangles as as shown in 
			\cref{fig:cell-type1}}
		\ELSE
		\STATE{Set Type as 'b'}
		\STATE{Find intersection points $z_3$--$z_6$ and form 4 triangles as as in \cref{fig:cell-type4}}
		\ENDIF
	\ELSIF{$\phi(z_1)<0$}
		\IF{ $\phi( (z_0 + z_1)/2 ) < 0$}
		\STATE{Set Type as 'c'}
		\STATE{Find intersection points $z_3$--$z_5$ and form 4 triangles as as shown in 
			\cref{fig:cell-type2}}
		\ELSE
		\STATE{Set Type as 'd'}
		\STATE{Find intersection points $z_3$--$z_6$ and form 2 triangles as as in \cref{fig:cell-type3}}
		\ENDIF
	\ELSE
		\STATE{Set Type as 'e'}
		\STATE{Find intersection point $z_3$ and form 1 triangle as in in \cref{fig:cell-type5}}
	\ENDIF
\ENDFOR
\end{algorithmic}
\end{algorithm}

\begin{figure}[h!]
    \centering
    \subfloat[]{\includegraphics[width=0.18\textwidth]{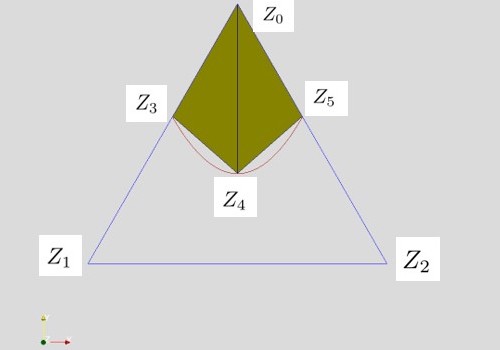} 
       \label{fig:cell-type1}}
                  \hfill
        \subfloat[]{\includegraphics[width=0.18\textwidth]{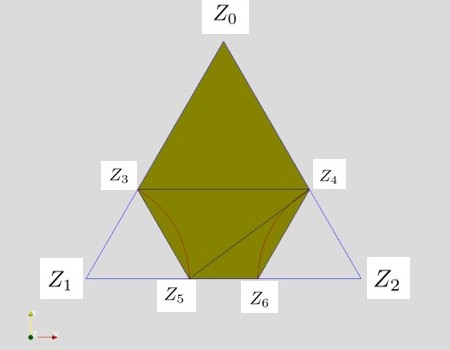}
        \label{fig:cell-type4}}
            \hfill
        \subfloat[]{\includegraphics[width=0.18\textwidth]{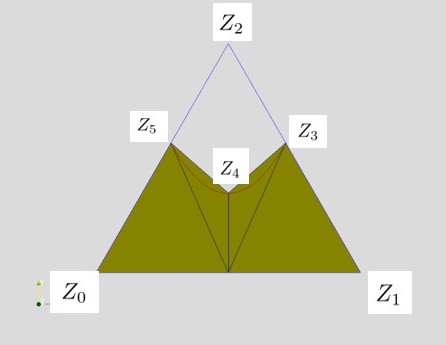}
        \label{fig:cell-type2}}
    \hfill
        \subfloat[]{\includegraphics[width=0.18\textwidth]{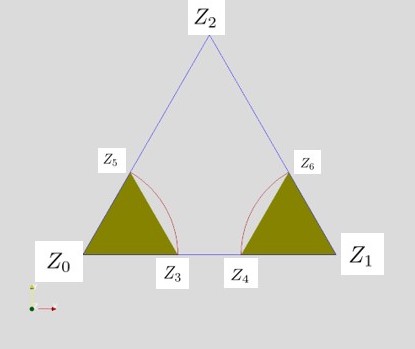}
        \label{fig:cell-type3}}
                   \hfill
        \subfloat[]{\includegraphics[width=0.18\textwidth]{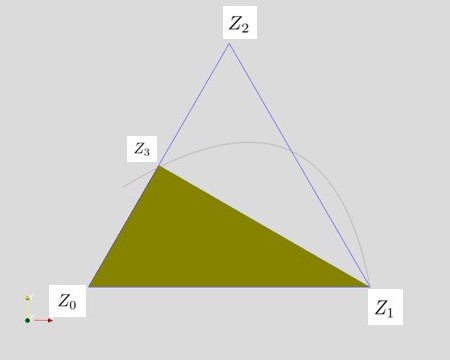}
        \label{fig:cell-type5}}
\caption{Partition of the intersection cell based on the cut of level set.}
\end{figure}  
\begin{figure}[h!]\label{fig:example-bc-mesh}
\centering
 \includegraphics[width=0.50\textwidth]{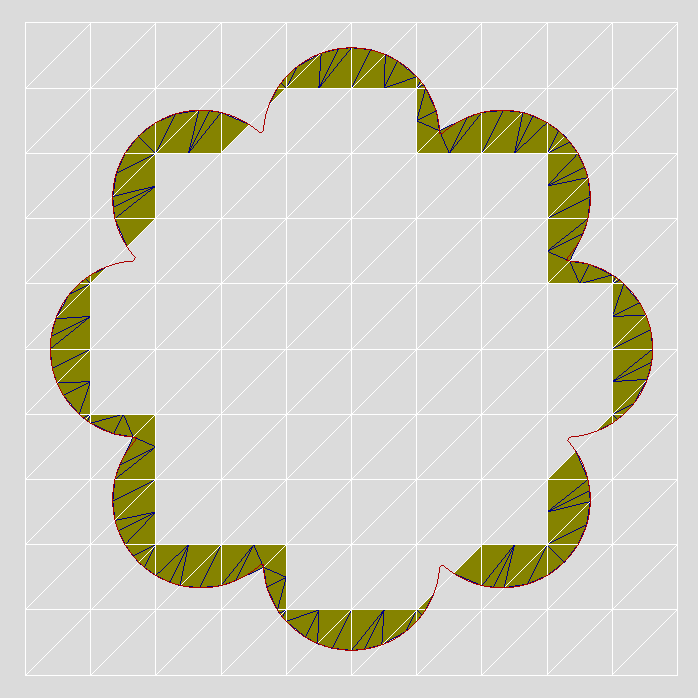}
 \caption{An example of the boundary correction mesh for Example 1}

\end{figure}
With the mesh $\cT^{bc}$ available,  we define $\tilde e$ to be  a piecewise linear conforming function with respect to $\cT_h^{bc}$ such that 
\[
	\tilde e(z) = \begin{cases}
	g(z) - u_h(z) & \mbox{if }  z \in \partial \O,\\[2mm]
	0 & \mbox{otherwise}.
	\end{cases}
\]
\begin{remark}
It is critical that $\tilde e$ is properly designed so that the
boundary correction does not over-estimate the error due to the
underresolved geometry. In our construction, it is easy to see that $\tilde e = \tilde g -u_h$ on $\partial \O$ where $\tilde g$ is the nodal interpolation of $g$ with respect to $\cT_h^{bc}$ on $\partial \O$.
Then we have that
\[
	\| \tilde g - u_h \|_{1/2,\partial \O} \le \| \nabla \tilde e\|_{\O} \lesssim \| \tilde g-u_h\|_{1/2,\partial \O}.
\]
The second inequality follows from the equivalence of norms.

\end{remark}

\subsection{Numerical Examples}
\begin{example} \label{ex:1}
The level set of the problem has a flower shape (see \cref{fig:example-bc-mesh})  that has the 
following representation:
\[
	\phi = \min(\phi_0, \phi_1, \cdots, \phi_8)
\]
with
\[
	\begin{cases}
	\phi_0(x,y)= x^2 + y^2 - r^2, & r=2\\
	\phi_i(x,y) = (x - x_i)^2 + (y - y_i)^2 - r_i^2, &r_i = \sqrt{2} r*(\sin( \pi/8) + \cos(\pi/8) ) \sin(\pi/8)
	\end{cases}
\]
for $i = 1, \cdots, 8$, $x_i = r(\cos(\pi/8)+\sin(\pi/8))\cos(i*\pi/4)$ and 
$y_i = r(\cos(\pi/8)+\sin(\pi/8))\cos(i*\pi/4)$.
The domain $\O$ is defined as $x \in \mathbb{R}^2$ such that $\phi(x) \le 0$. 
The data are given such that $g = 0$ on $\partial \O$ and 

\[   
	f(x,y) =\left\{
	\begin{array}{lll}
      	10 & \mbox{if }  (x - x_1)^2 + (y - y_1)^2 \le r_1^2/2.\\
	0 & \mbox{otherwise}.
\end{array} 
\right. \]
\end{example}

In the numerical scheme, we take $g_h=0$.
With the stopping criteria that the total number of degree of freedoms be not greater than $7000$, the final meshes obtained without and with adding the boundary correction term are given $\| \nabla \tilde e\|$  in 
\cref{fig:mesh1-1} and \cref{fig:mesh1-2}, respectively. We observe that slightly more degree of freedoms are added around the concave corners  in \cref{fig:mesh1-2}.
The corresponding  convergence performances of each adaptive procedure
are given in 
\cref{fig:ex1-error-comparison-bv}. We observe that both estimators converge optimally when the mesh become fine enough. Adding the boundary correction term does slightly increases the estimator at the initial stage. However, the weight is diminishing as the mesh gets finer.

%
\begin{figure}
    \centering
    \subfloat[]{\includegraphics[width=0.45\textwidth]
    {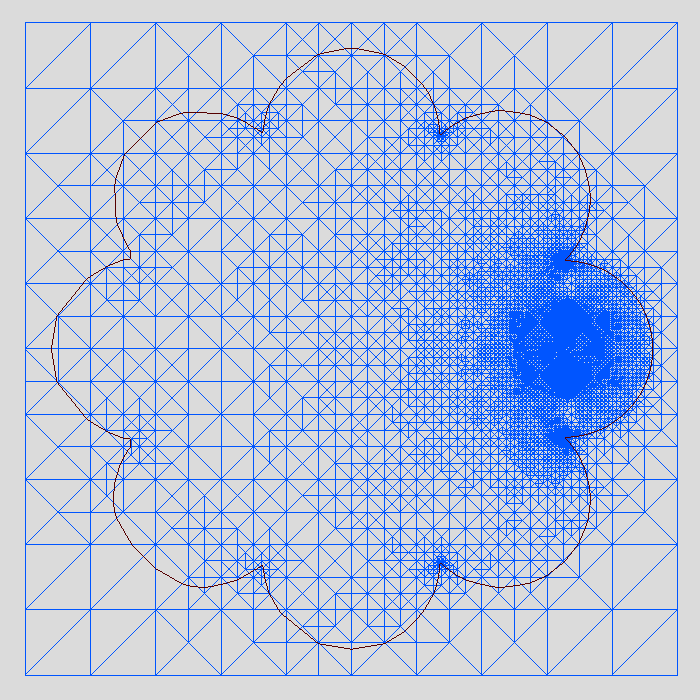} 
        \label{fig:mesh1-1}}
    \hfill
        \subfloat[]{\includegraphics[width=0.45\textwidth]{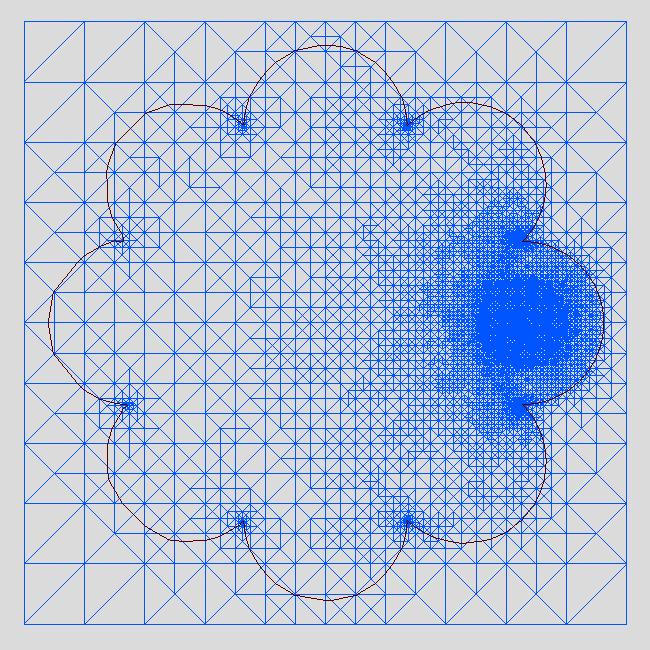}
        \label{fig:mesh1-2}}
\caption{Final meshes generated without and with $\| \nabla \tilde e\|_K$.}
\end{figure}

\begin{figure}[h!]
\centering
\includegraphics[width=0.6\textwidth]{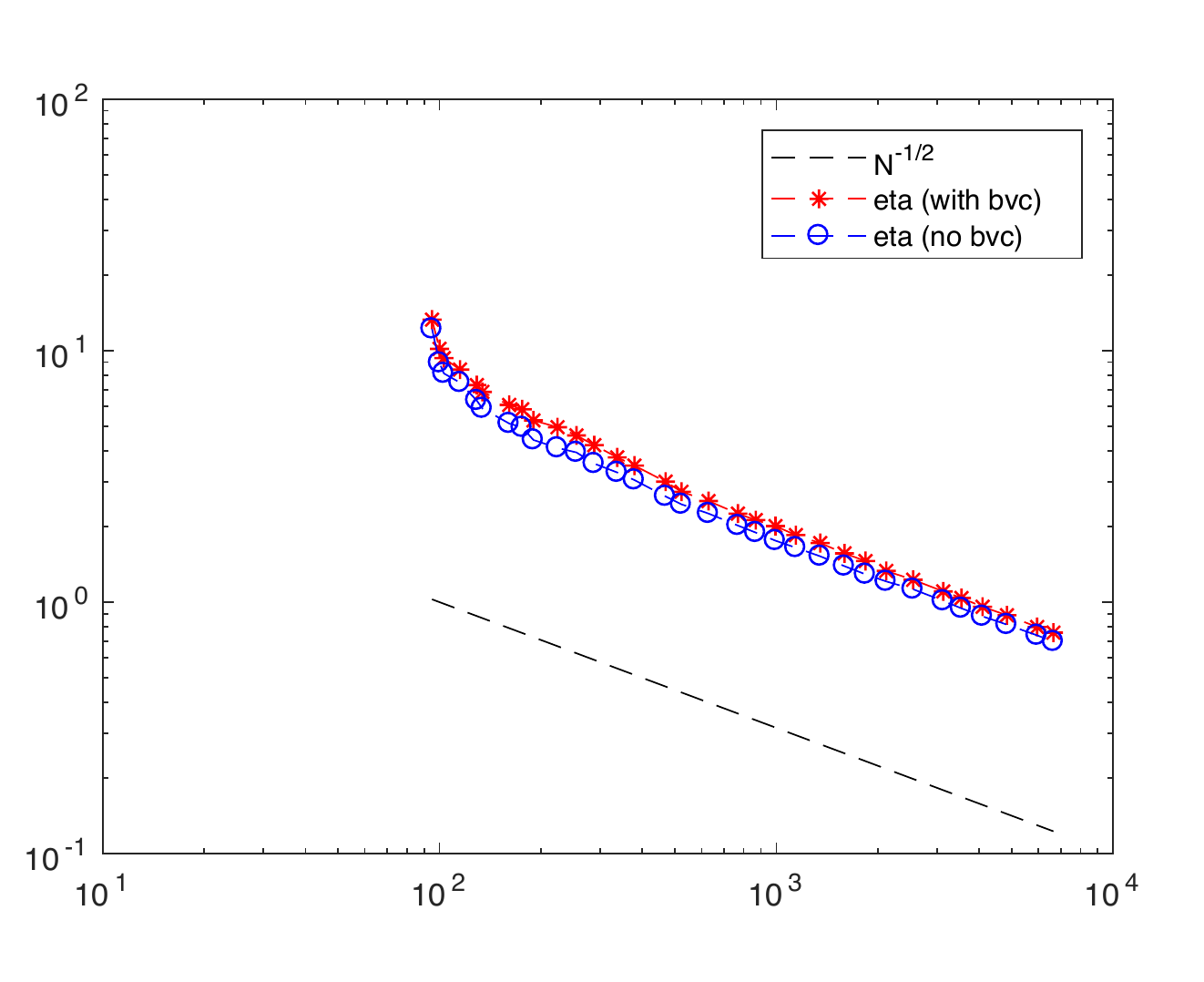}
    \label{fig:ex1-error-comparison-bv}    
\caption{Estimator convergence rate without and with $\| \nabla \tilde e\|$}
\end{figure}

\begin{example}\label{ex:2}
The level set of this problem has the following representation:
\[
	\phi = \max(\phi_0, -\phi_1, \cdots, -\phi_8)
\]
with $\phi_0$ to $\phi_8$ defined the same as in \cref{ex:1}. The datum are chosen such that 
\[f =0 \quad \mbox{and} \quad g = y^2.\]
In the numerical scheme we approximate the boundary by $\O_h$ defined in \cref{domain-approx}. For the Dirichlet date we take $g_h$ to be the conforming piecewise linear interpolation of $g =y^2$.
With the stopping criteria that the total number of degree of freedoms be not greater than $7000$, the final meshes obtained without and with adding the boundary correction term $\| \nabla \tilde e\|_K$  are given in 
 \cref{fig:ex2-mesh1-1} and \cref{fig:ex2-mesh1-2}, respectively.  The
 corresponding  convergence performance of each adaptive procedure are
 given in 
\cref{fig:ex2-error-comparison-bv}. We observe that in both cases the estimator converges optimally. In \cref{fig:ex2-mesh1-2}  note that there is no more obvious dense refinement comparing to \cref{fig:ex2-mesh1-1}  around the boundary including the corners since in this case the approximation error $g-g_h$ is of uniform order everywhere.

\begin{figure}
    \centering
    \subfloat[]{\includegraphics[width=0.46\textwidth]{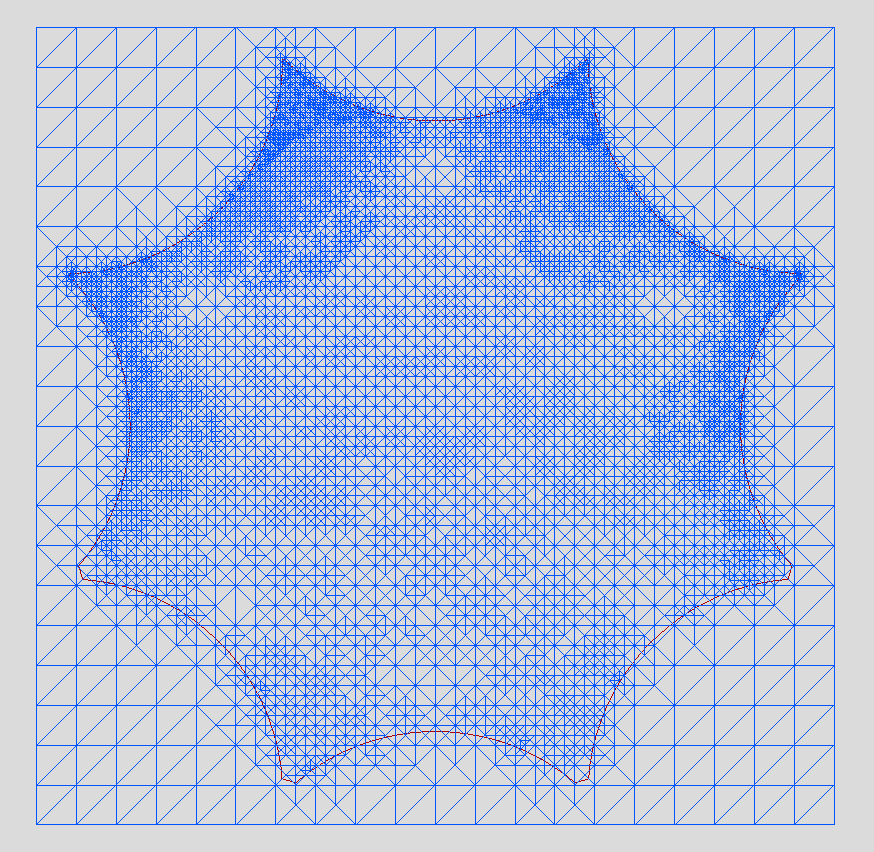} 
        \label{fig:ex2-mesh1-1}
}
     \subfloat[]{\includegraphics[width=0.45\textwidth]{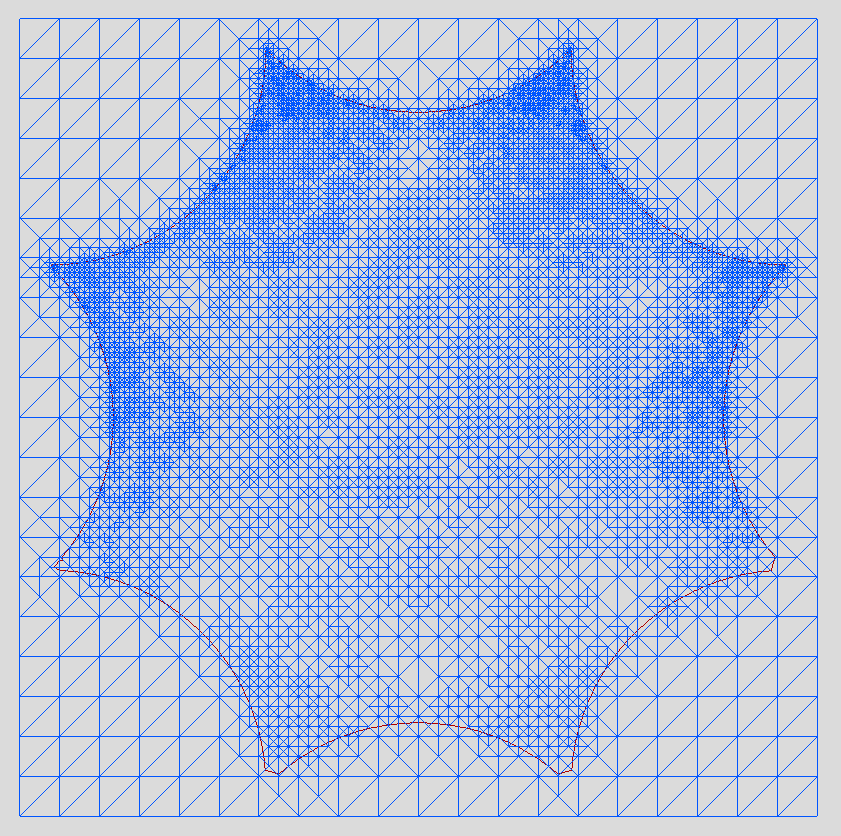}
        \label{fig:ex2-mesh1-2}
}
\caption{Final meshes generated without and with $\| \nabla \tilde e\|_K$.}
\end{figure}

\begin{figure}
\centering
\includegraphics[width=0.6\textwidth]{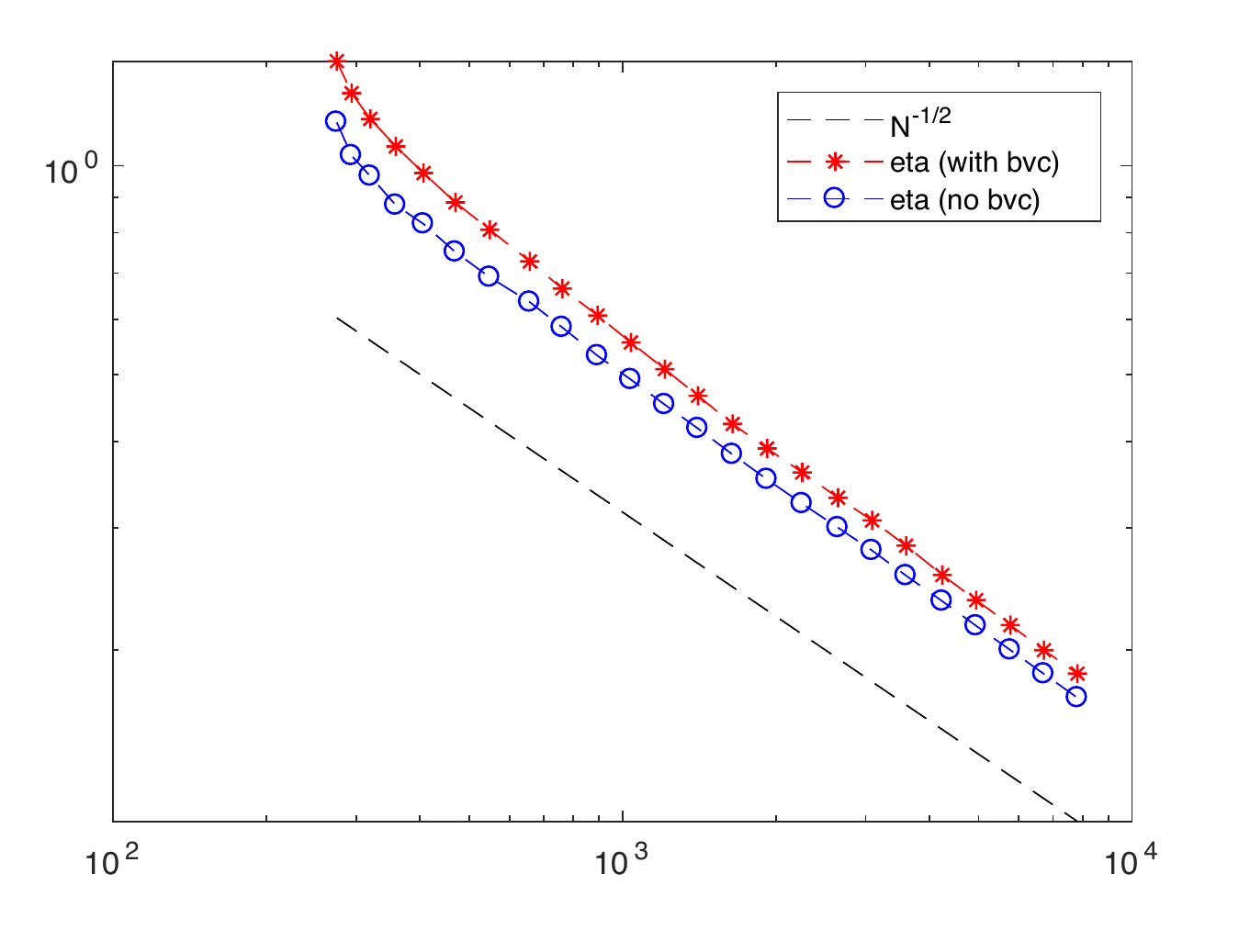}
    \label{fig:ex2-error-comparison-bv} 
\caption{estimator convergence rate without and with $\| \nabla \tilde e\|$}
\end{figure}
\end{example}

\subsection{Example 3}
In this example we consider the reentrant problem whose solution has the following polar representation:
\[
	u(r, \theta) = r^{\alpha} \sin(\alpha \theta),
\]
with $\alpha = \pi/ \omega$ and $\omega$ being the angle of the reentrant corner. In
this example, we test two values for $\omega$, i.e., $31\pi/8$ and
$63\pi/16$. The stopping criteria is set such that the maximal number
of degrees of freedom does not exceed $5000$.
In the numerical scheme we approximate the Dirichlet datum $g$ by $g_h$ using conforming piecewise linear interpolation.
The final meshes generated without adding the boundary correction error are given in 
\cref{fig:ex3-mesh1}--\cref{fig:ex3-mesh2} and with the boundary correction error are given in 
\cref{fig:ex3-mesh3}--\cref{fig:ex3-mesh4}. 
The corresponding convergence rate of estimators  are presented in 
\cref{fig:ex3-error3}--\cref{fig:ex3-error4}. 
We again observe the optimal convergence performance for the estimator and true error in all cases. This indicates that the estimator, with or without the boundary correction error, works equivalently effective for problem even with singularity on the boundary.

\begin{figure}[h!]
    \centering
    \subfloat[]{\includegraphics[width=0.45\textwidth]
    {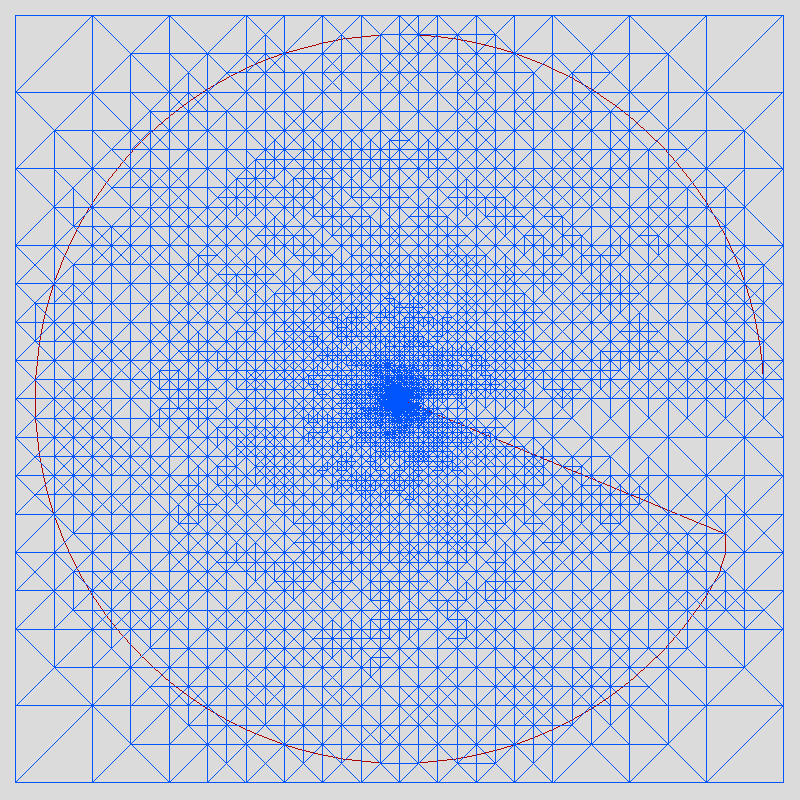} 
        \label{fig:ex3-mesh1}}
    \hfill
        \subfloat[]{\includegraphics[width=0.45\textwidth]{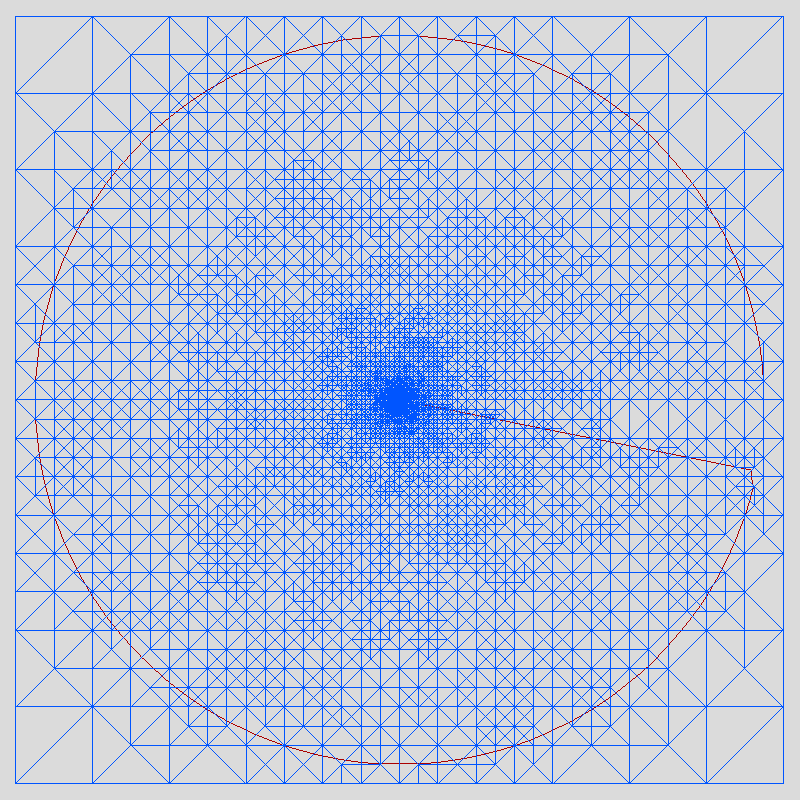}
        \label{fig:ex3-mesh2}}
\caption{Final meshes generated without $\| \nabla \tilde e\|_K$. Left: $\omega=31\pi/8$. Right:  $\omega=63\pi/16$.}
\end{figure}


\begin{figure}[h!]
    \centering
    \subfloat[]{\includegraphics[width=0.45\textwidth]
    {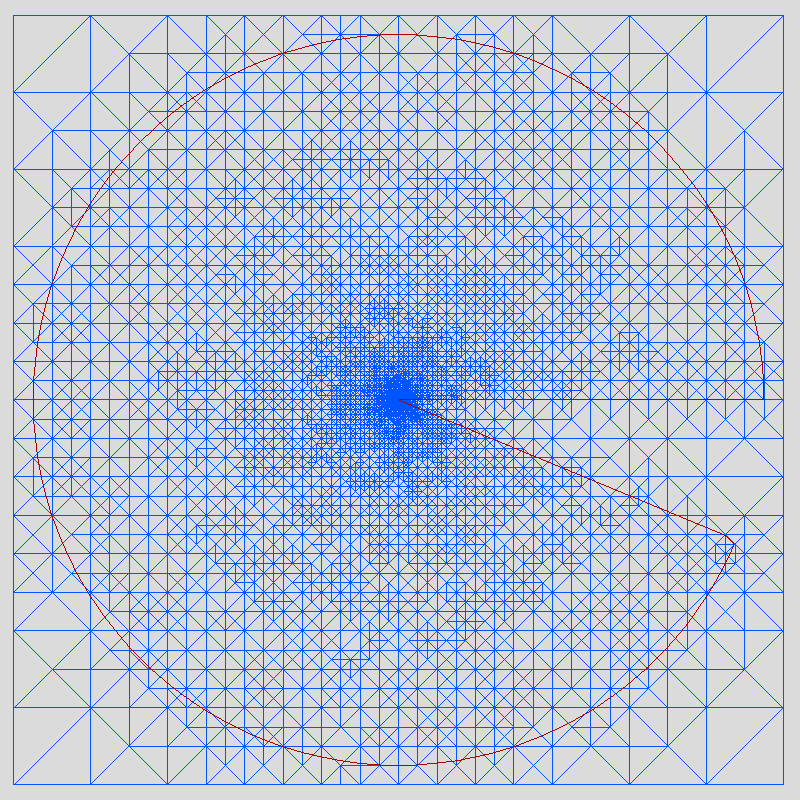} 
        \label{fig:ex3-mesh3}}
    \hfill
        \subfloat[]{\includegraphics[width=0.45\textwidth]{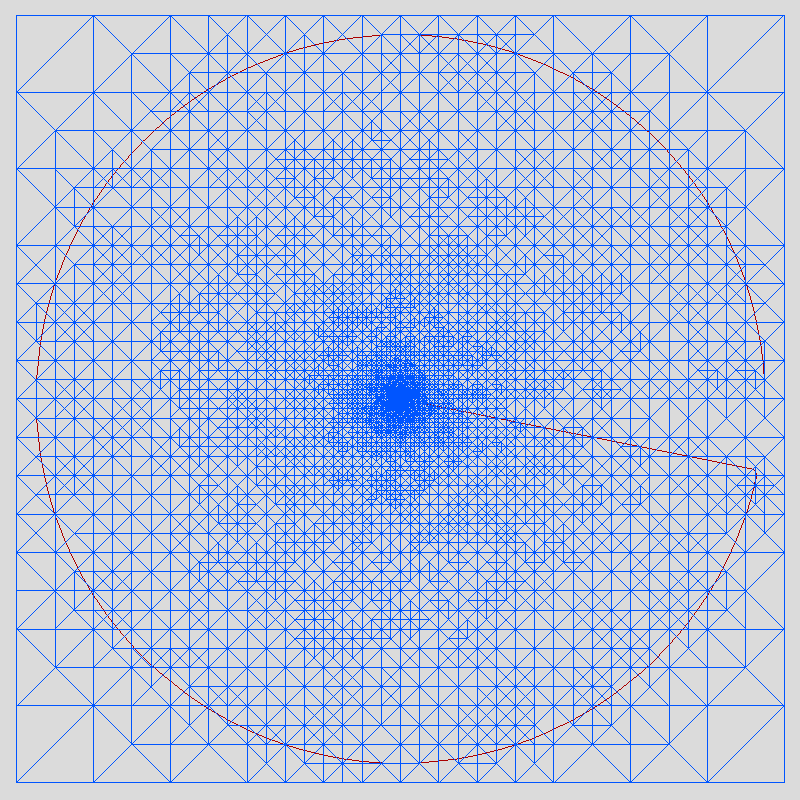}
        \label{fig:ex3-mesh4}}
\caption{Final meshes generated with $\| \nabla \tilde e\|_K$.  Left: $\omega=31\pi/8$. Right:  $\omega=63\pi/16$.}
\end{figure}

\begin{figure}[h!]
    \centering
    \subfloat[]{\includegraphics[width=0.48\textwidth]
    {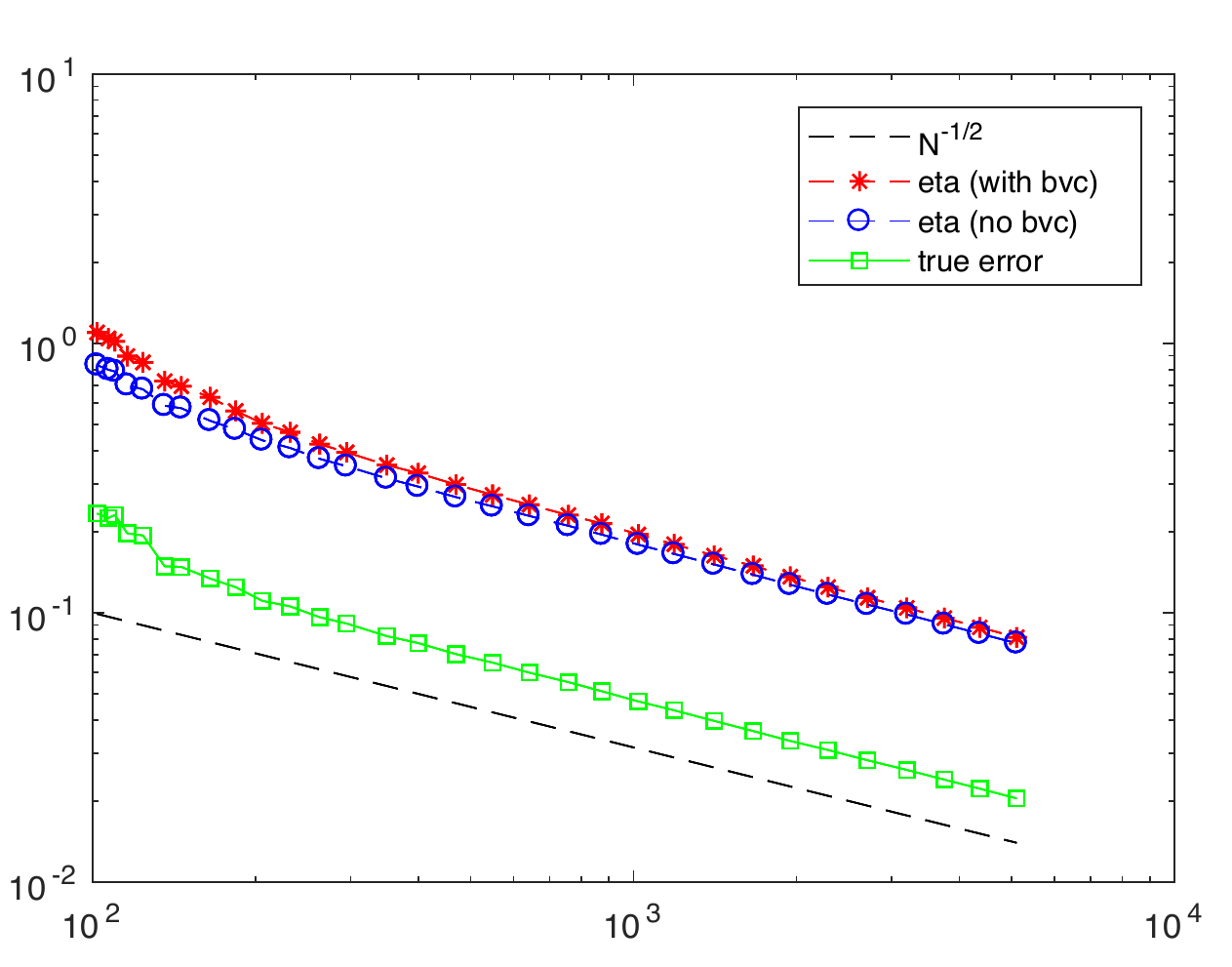} 
        \label{fig:ex3-error3}}
        \subfloat[]{\includegraphics[width=0.45\textwidth]{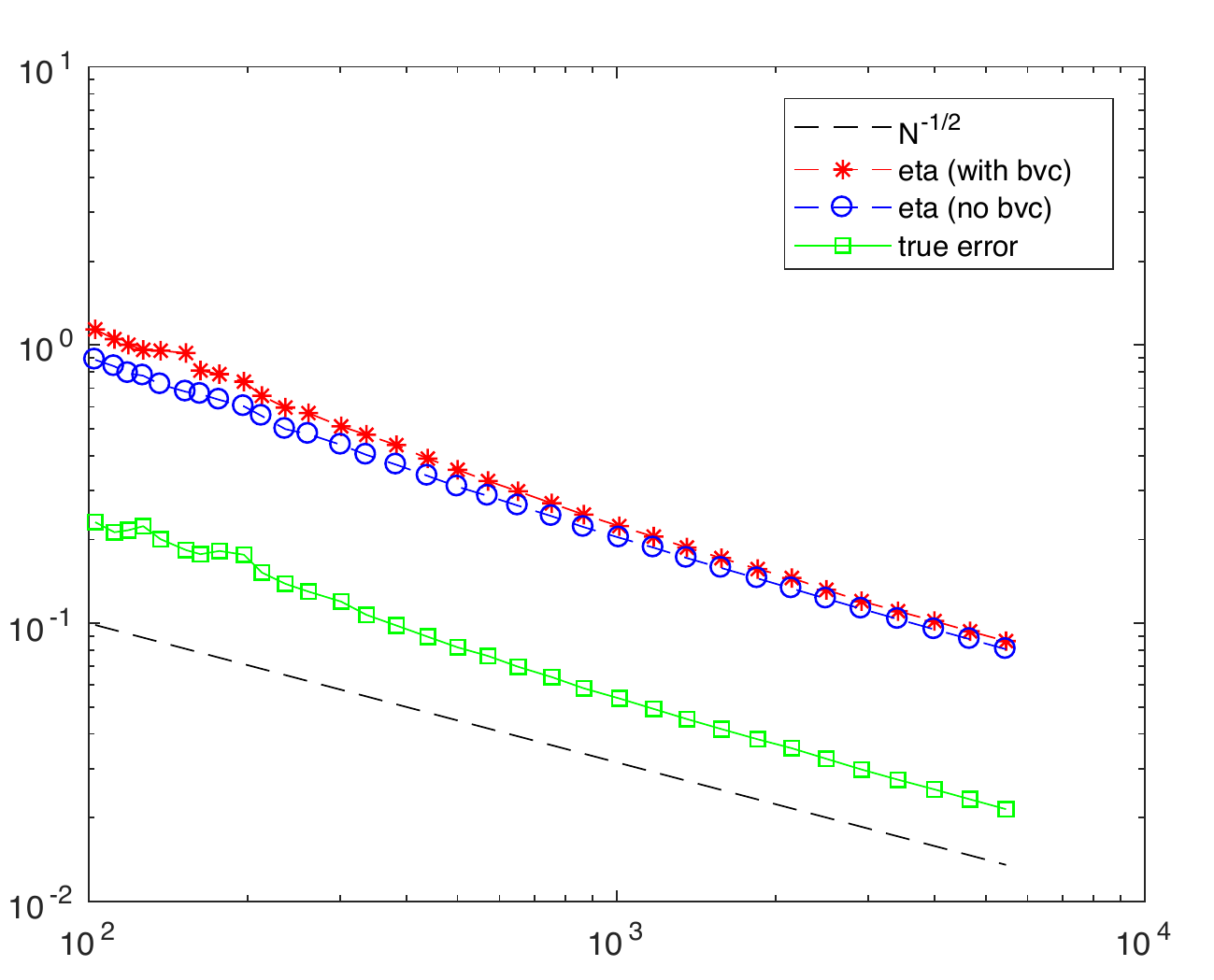}
        \label{fig:ex3-error4}}
\caption{Estimator convergence performance without and with $\| \nabla \tilde e\|$. Left: $\omega=31\pi/8$. Right:  $\omega=63\pi/16$.}
\end{figure}

We now test the same procedure but with $g_h$ to be the piecewise
constant interpolation of $g$ for $w = 31 \pi/16$. From the final meshes generated in \cref{fig:ex3-mesh5} and \cref{fig:ex3-mesh6} we observe that more degree of freedoms are added around the boundary in both cases due to the poorer approximation of the boundary data. The estimators, however, still converges optimally in both cases when the meshes are fine enough.

\begin{figure}[h!]
    \centering
    \subfloat[]{\includegraphics[width=0.48\textwidth]
    {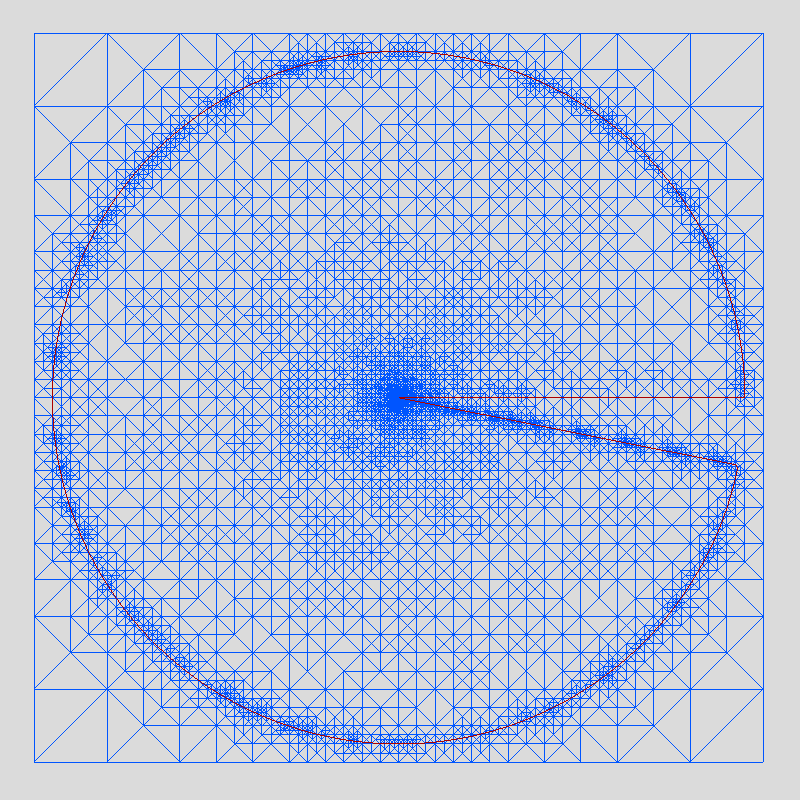} 
        \label{fig:ex3-mesh5}}
        \subfloat[]{\includegraphics[width=0.47\textwidth]{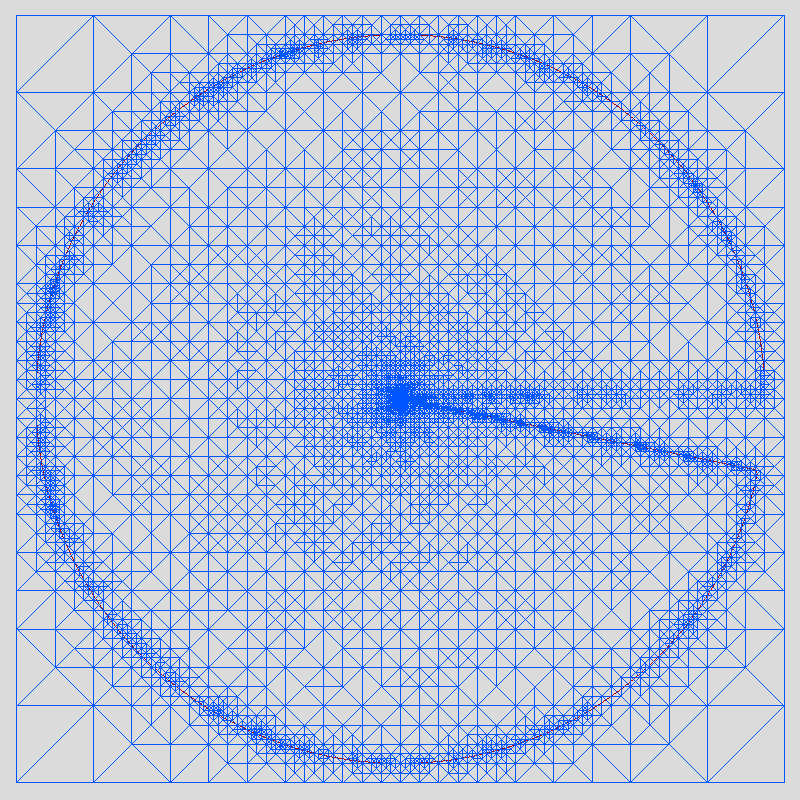}
        \label{fig:ex3-mesh6}}
\caption{Final meshes without and with $\| \nabla \tilde e\|$.}
\end{figure} 

\begin{figure}[h!]
    \centering
\includegraphics[width=0.6\textwidth]{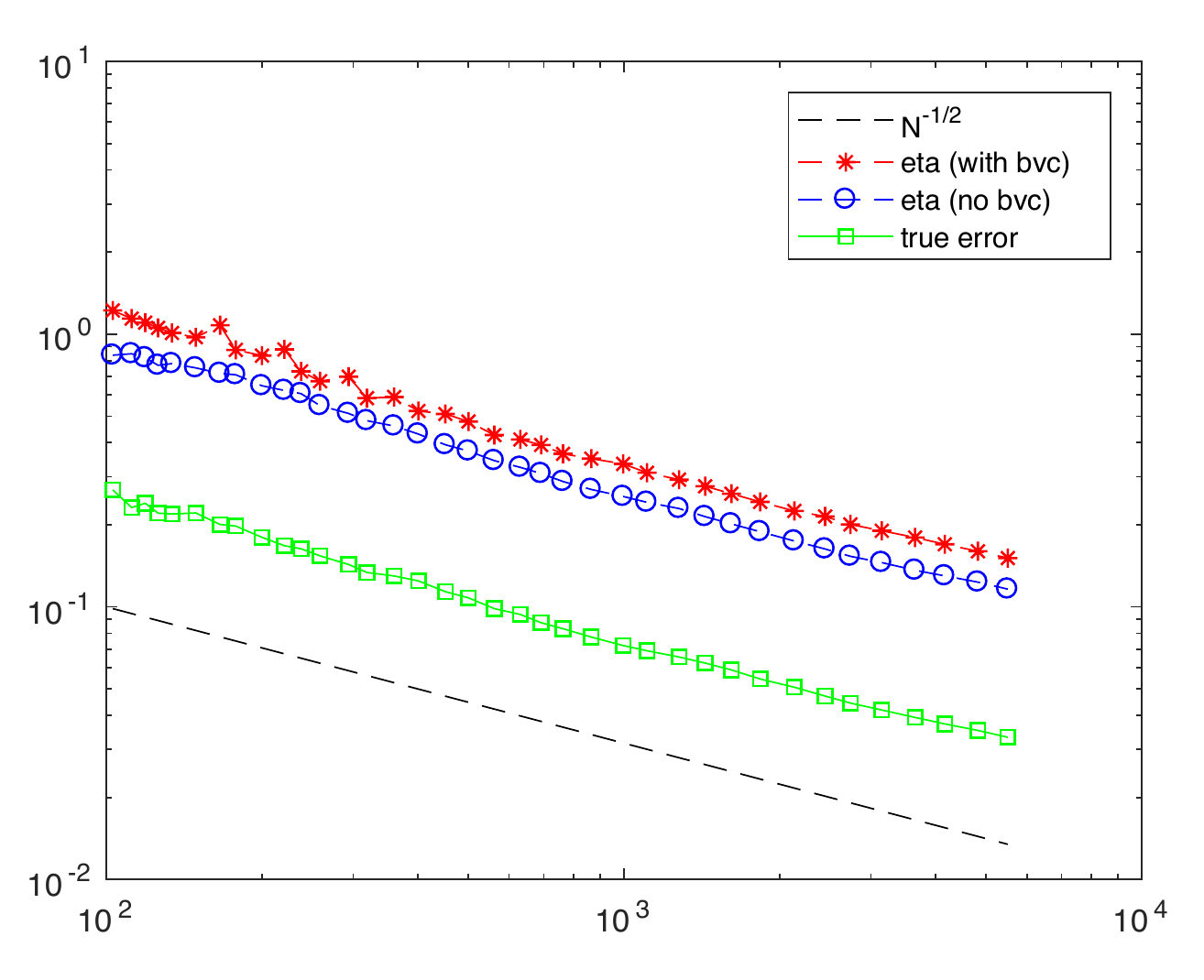}
        \label{fig:ex3-error6}
\caption{Convergence performance for estimator without and with boundary correction.}
\end{figure}

\subsection{Example 4}
In this example, we consider the problem that has the following representation:
\[
	u(x,y)= r^{\alpha} \sin(\alpha \theta)
	+ exp(100* ((x - 0.5)^2+ (y - 0.5)^2))
\]
with $r = \sqrt{x^2 + y^2}$, $\theta = \tan^{-1}(y/x) $, $\alpha = \pi/ \omega$ and $\omega = 31\pi/16$. 
This problem has multiple singularities with one reentrant corner at $(0.0)$ and one peak at $(0.5, 0.5)$.
In the numerical scheme, $g_h$ and $f_h$ are taken as the linear nodal interpolation of $g$ and $f$, respectively.  
The stopping criteria is set to not exceed the maximal refinement step
of $50$ and the maximal  number of degrees of freedom of $7500$.
The final meshes are given in 
\cref{fig:ex4-mesh1} and \cref{fig:ex4-mesh2} and
 their corresponding convergence performance are presented in 
\cref{fig:ex4-error2}. 

\begin{figure}[ht]
    \centering
    \subfloat[]{\includegraphics[width=0.40\textwidth]
    {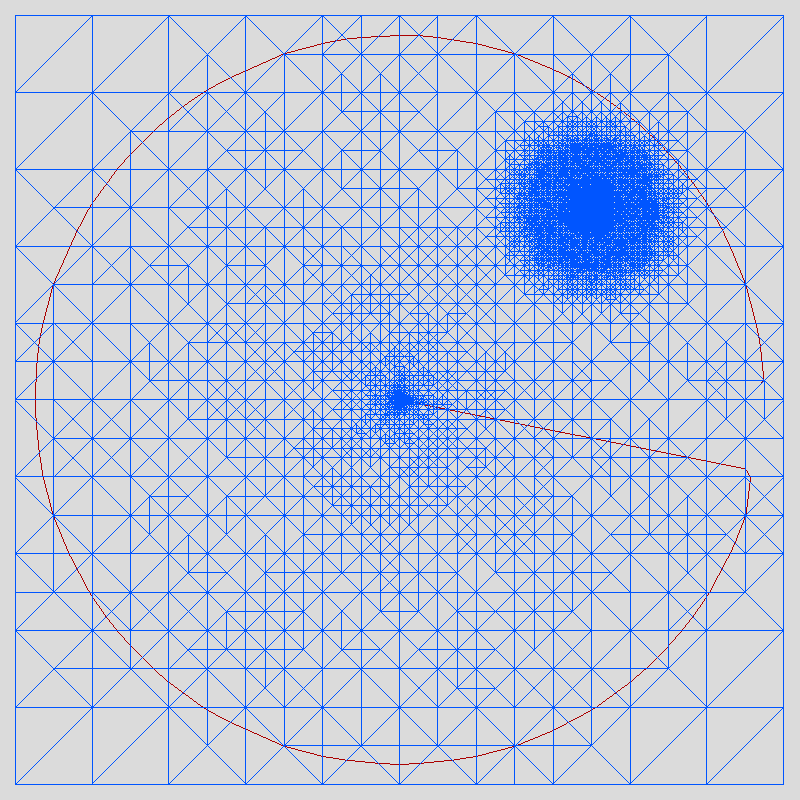} 
        \label{fig:ex4-mesh1}}
    \hfill
        \subfloat[]{\includegraphics[width=0.40\textwidth]{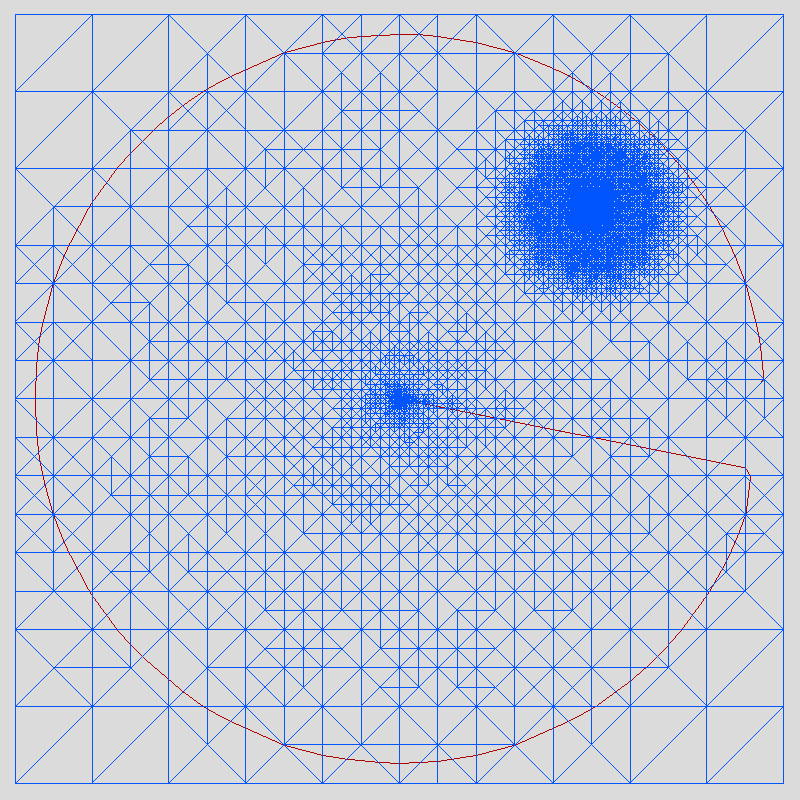}
        \label{fig:ex4-mesh2}}
\caption{Final meshes generated without (left) and with (right) $\| \nabla \tilde e\|_K$.}
\end{figure}

\begin{figure}[ht]
    \centering
\includegraphics[width=0.6\textwidth]{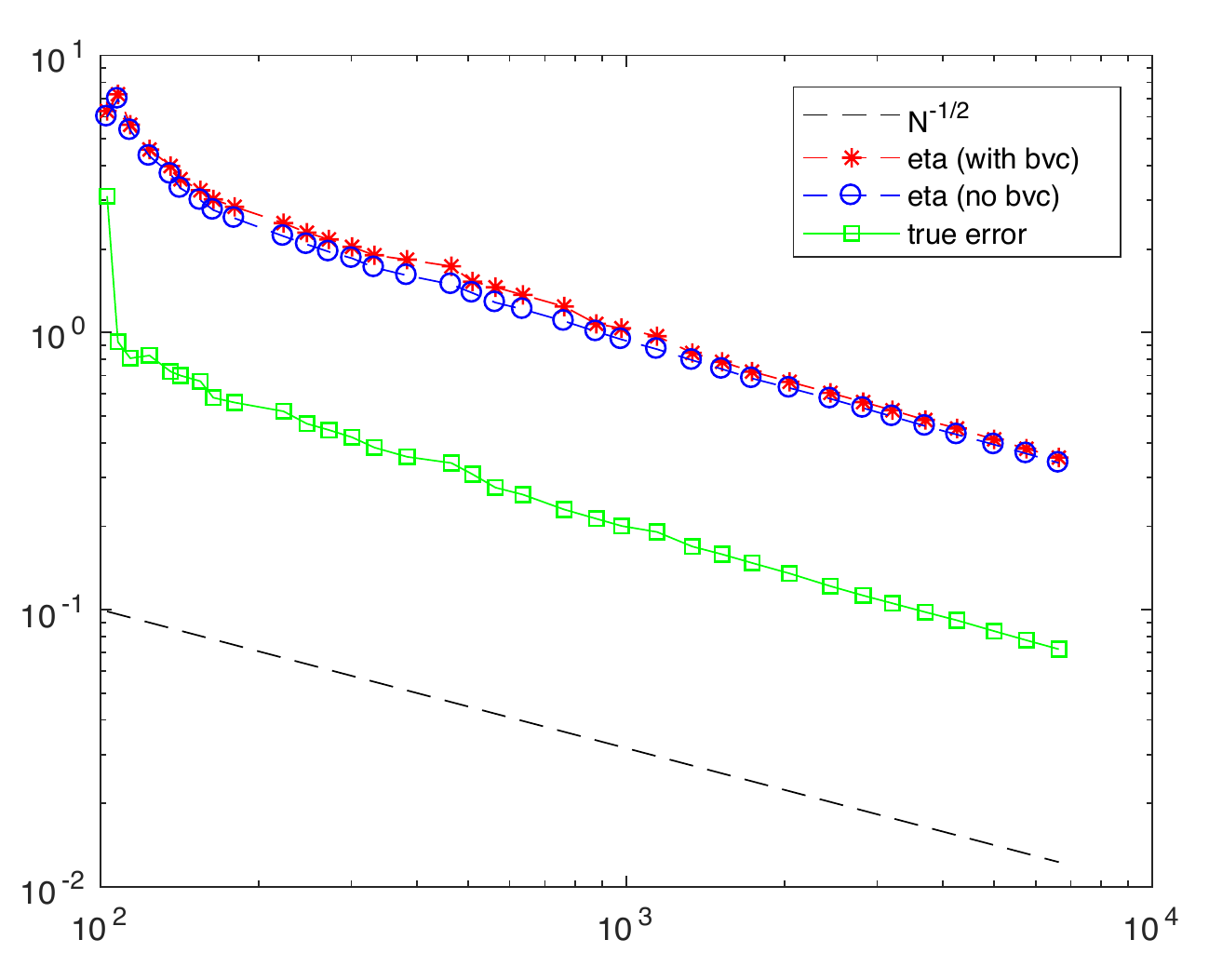}
        \label{fig:ex4-error2}
\caption{Convergence performance for estimator without and with $\| \nabla \tilde e\|$.}
\end{figure}

This example shows that the algorithm works also effectively when multiple singularities occurs, no matter the singularity happens on the boundary or inside the domain.

\paragraph{Acknowledgement} EB and CH were supported by EPSRC, UK,
Grant No. \\ EP/P01576X/1. ML was supported by the Swedish Foundation for Strategic Research Grant No.\ AM13-0029, the Swedish Research Council Grants No. 2017-03911, and the Swedish strategic research programme eSSENCE.


\bibliography{BDT}
\end{document}